\pdfoutput=1

\documentclass[11pt]{article}
\usepackage[utf8]{inputenc}
\usepackage{amssymb,amsthm}
\usepackage{mathtools}
\usepackage[usenames,dvipsnames]{xcolor}

\usepackage{hyperref}
\hypersetup{colorlinks=true,linkcolor={Brown},citecolor={Brown},urlcolor={Brown},bookmarksdepth=2}
\usepackage[compress]{cleveref}
\usepackage{thmtools, thm-restate}
\declaretheoremstyle[%
  spaceabove=.8\baselineskip,%
  spacebelow=.8\baselineskip,%
  headfont=\bfseries,%
  notefont=\normalfont,%
  bodyfont=\itshape,%
  postheadspace=.5em%
]{thms}

\declaretheoremstyle[%
  spaceabove=.8\baselineskip,%
  spacebelow=.8\baselineskip,%
  headfont=\bfseries,%
  notefont=\normalfont,%
  bodyfont=\normalfont,%
  postheadspace=.5em%
]{defn}

\numberwithin{equation}{section}
\theoremstyle{thms}
\newtheorem{thrm}[equation]{Theorem}
\newtheorem{cor}[equation]{Corollary}
\newtheorem{lem}[equation]{Lemma}
\newtheorem{prop}[equation]{Proposition}

\theoremstyle{defn}

\newtheorem{defn}[equation]{Definition}
\newtheorem*{defn*}{Definition}
\newtheorem{exa}[equation]{Example}

\newtheorem{notn}[equation]{Convention}
\newtheorem*{notn*}{Notation}
\newtheorem{rmk}[equation]{Remark}
\newtheorem{hyp}[equation]{Hypothesis}

\crefname{prop}{Proposition}{Propositions}
\crefname{hyp}{Hypothesis}{Hypotheses}
\crefname{cor}{Corollary}{Corollaries}
\crefname{lem}{Lemma}{Lemmata}
  \crefname{cns}{Construction}{Constructions}
  \crefname{def}{Definition}{Definitions}
  \crefname{rmk}{Remark}{Remarks}
  \crefname{notn}{Convention}{Conventions}
  \crefname{exa}{Example}{Examples}
  \crefname{thrm}{Theorem}{Theorems}
  \crefname{thrmprime}{Theorem'}{Theorem'}
\usepackage[shortlabels,inline]{enumitem}
  \newlist{axenum}{enumerate}{1} 
  \setlist[axenum]{label=(\Alph*)}
  \crefname{axenumi}{Axiom}{Axioms}
\setlist{itemsep=.1em}
\usepackage{amsmath}
\usepackage[full]{textcomp} 
\usepackage[p,osf]{fbb} 
\usepackage[scaled=.95,type1]{cabin} 
\usepackage[varqu,varl]{zi4}
\usepackage[T1]{fontenc} 
\usepackage[libertine]{newtxmath}
\usepackage[cal=boondoxo,bb=boondox,frak=boondox]{mathalfa}

\usepackage{textcomp}
\usepackage{tikz-cd}
\RequirePackage[margin=3.6cm]{geometry}
\usetikzlibrary{arrows,arrows.meta}
\tikzcdset{arrow style=tikz}

\newcommand\footnotewomarker[1]{%
  \begingroup
  \renewcommand\thefootnote{}\footnote{#1}%
  \addtocounter{footnote}{-1}%
  \endgroup
}

\usepackage[english]{babel}
\usepackage{csquotes}

\usepackage[
backend = biber,
style = alphabetic,
citestyle = alphabetic,
maxbibnames = 99,
url = false]{biblatex}

\usepackage{xspace}
\usepackage{xcolor}
\newcommand{\nc}{\newcommand}
\nc{\BG}{\textup{B}G}

\nc{\dmo}{\DeclareMathOperator}
\nc{\martin}[1]{\textcolor{red}{#1}}
\dmo{\id}{id}
\nc{\kk}{k}
\nc{\base}{S} 
\nc{\V}{\mathcal{S}}
\nc{\Var}[1][]{\mathrm{Var}_{\ifblank{#1}{\kk}{#1}}}
\nc{\Varqp}[1][]{\mathrm{Var}_{\ifblank{#1}{\kk}{#1}}^{\mathrm{qp}}}
\nc{\Sub}[1][]{\mathrm{Sub}_{\ifblank{#1}{X}{#1}}}
\dmo{\DM}{DM}
\dmo{\MM}{MM}
\dmo{\DMls}{DM_{ls}}
\dmo{\DAet}{DA^{\acute{e}t}}
\dmo{\DMc}{DM_c}
\dmo{\VMHSadm}{VMHS^{adm}}
\nc{\Dbc}{\mathrm{D}^{\mathrm{b}}_{\mathrm{c}}}
\nc{\Dbqc}{\mathrm{D}^{\mathrm{b}}_{\mathrm{qc}}}
\nc{\Dbgm}{\mathrm{D}^{\mathrm{b}}_{\mathrm{gm}}}
\nc{\Dbh}{\mathrm{D}^{\mathrm{b}}_{\mathrm{h}}}
\nc{\Dbrh}{\mathrm{D}^{\mathrm{b}}_{\mathrm{rh}}}
\nc{\Dbls}{\mathrm{D}^{\mathrm{b}}_{\mathrm{ls}}}
\nc{\Dbgmls}{\mathrm{D}^{\mathrm{b}}_{\mathrm{gm,ls}}}
\nc{\Dbmhm}[1]{\mathrm{D}^{\mathrm{b}}\textup{MHM}(#1)}
\nc{\Dperf}{\mathrm{D}^{\mathrm{perf}}}
\dmo{\MHM}{MHM}
\dmo{\Perv}{Perv}
\dmo{\Pervgm}{Perv_{gm}}
\nc{\Dblsmhm}[1]{\mathrm{D}^{\mathrm{b}}_{\mathrm{ls}}\textup{MHM}(#1)}
\nc{\rat}{\mathrm{rat}}
\dmo{\LS}{LS}
\dmo{\LSgm}{LS_{gm}}
\dmo{\Cons}{Cons}
\dmo{\iC}{DE}
\nc{\hDmod}[1]{\cD_{#1}\operatorname{-mod_h}}
\nc{\cDmod}[1]{\cD_{#1}\operatorname{-mod_c}}
\dmo{\Sh}{Sh}
\dmo{\Rep}{Rep}
\nc{\dL}{\mathbf{L}}
\nc{\dR}{\mathbf{R}}
\nc{\et}{\textup{\'et}}
\nc{\proet}{\textup{pro\'et}}
\nc{\pr}{\mathcal{P}}
\dmo{\Supp}{Supp}
\dmo{\supp}{supp}
\dmo{\cone}{cone}
\dmo{\im}{im}
\nc{\cons}[1]{#1_{\mathrm{cons}}}
\nc{\zar}[1]{#1_{\mathrm{Zar}}}
\dmo{\Idl}{Idl}
\nc{\ttCat}{\mathsf{ttCat}}
\nc{\tCat}{\mathsf{Cat}^{\otimes}}
\nc{\CohSp}{\mathsf{CohSp}}
\nc{\DLat}{\mathsf{DLat}}
\nc{\TopSpc}{\mathsf{Top}}
\dmo{\cl}{cl}
\dmo{\CS}{C}
\dmo{\CSp}{D}
\dmo{\CSls}{\CS^{\textup{ls}}}
\dmo{\CSlsh}{\CS^{\textup{ls},\heartsuit}}
\dmo{\CSpls}{\CSp^{\textup{ls}}}
\nc{\radtt}[1]{\sqrt{\textup{tt}}(#1)}
\nc{\ttid}[1]{\textup{tt}(#1)}
\nc{\cA}{\mathcal{A}}
\nc{\cC}{\mathcal{C}}
\nc{\cD}{\mathcal{D}}
\nc{\cI}{\mathcal{I}}
\nc{\cK}{\mathcal{K}}
\nc{\cL}{\mathcal{L}}
\nc{\cM}{\mathcal{M}}
\nc{\cO}{\mathcal{O}}
\nc{\cP}{\mathcal{P}}
\nc{\cQ}{\mathcal{Q}}
\nc{\cR}{\mathcal{R}}
\nc{\cT}{\mathcal{T}}
\nc{\cU}{\mathcal{U}}
\nc{\cV}{\mathcal{V}}
\nc{\mfm}{\mathfrak{m}}
\nc{\mfp}{\mathfrak{p}}
\nc{\mfq}{\mathfrak{q}}
\nc{\set}[1]{\left\{\,#1\,\right\}}
\dmo{\Der}{D}
\nc{\Db}{\Der^{\textup{b}}}
\dmo{\opname}{op}
\dmo{\modf}{mod}
\nc{\op}{^{\opname}}
\nc{\into}{\mathop{\rightarrowtail}}
\nc{\onto}{\mathop{\twoheadrightarrow}}
\nc{\xto}[1]{\xrightarrow{#1}}
\nc{\isoto}{\xto{\sim}}
\nc{\isofrom}{\xfrom{\sim}}
\nc{\xfrom}[1]{\overset{#1}{\leftarrow}}
\nc{\xinto}[1]{\overset{#1}{\,\into\,}}
\nc{\xonto}[1]{\overset{#1}{\,\onto\,}}
\usepackage{bbm}
\nc{\bool}{\mathbb{2}}
\nc{\oid}{\mathbb{u}}
\nc{\cid}{\mathbb{z}}
\nc{\lcid}{\mathbb{v}}
\nc{\OID}{\mathbb{U}}
\nc{\CID}{\mathbb{Z}} 

\dmo{\SI}{\mathbb{S}} 
\dmo{\TI}{\mathbb{T}} 
\dmo{\TIc}{\mathbb{T}^c} 
\dmo{\RI}{\mathbb{R}} 
\dmo{\RIc}{\mathbb{R}^c} 
\dmo{\Op}{\Omega} 

\nc{\affine}{\mathbb{A}}
\nc{\DD}{\mathbb{D}}
\nc{\QQ}{\mathbb{Q}}
\nc{\CC}{\mathbb{C}}
\nc{\EE}{\mathbb{E}}
\nc{\FF}{\mathbb{F}}
\nc{\XX}{\mathbb{X}}
\nc{\YY}{\mathbb{Y}}
\nc{\ZZ}{\mathbb{Z}}
\nc{\one}{\mathbb{1}}
\dmo{\Spc}{Spc}
\dmo{\Spcs}{Spc^{\wedge}}
\dmo{\Spec}{Spec}
\dmo{\Hm}{H}

\nc{\loccit}{\textsl{loc.\,cit}.\@\xspace}

\nc{\fun}[3][]{%
  \ifblank{#1}{
    \operatorname{Fun}(#2,#3)}
  {
    \operatorname{Fun}^{#1}(#2,#3)}}
\nc{\Mod}[2][]{%
  \ifblank{#1}{
    \operatorname{Mod}(#2)}
  {
    \operatorname{Mod}_{#1}(#2)}}

\makeatletter
\def\namedlabel#1#2{\begingroup
  \def\@currentlabel{#2}%
  \textbf{#2}
   \label{#1}\endgroup
}
\makeatother

\nc{\bref}[1]{\textbf{\ref{#1}}}

\newcommand{\kw}{constructible sheaves, holonomic D-modules, mixed Hodge modules, motivic sheaves, constructible systems, support datum, tensor-triangular geometry, smashing spectrum, classification, reconstruction}
\newcommand{\ttl}{Supports for constructible systems}
\title{\ttl}
\AtEndPreamble{\hypersetup{
    pdfcreator    = {\LaTeX{}},
    pdfencoding   = auto,
    psdextra,
    pdfauthor     = {Martin Gallauer},
    pdftitle      = {\ttl},
    pdfsubject    = {\ttl},
    pdfkeywords   = {\kw}}}
\begin{document}
\author{Martin Gallauer\footnotewomarker{Mathematical Institute, University of Oxford, UK (gallauer@maths.ox.ac.uk); Max Planck Institute for Mathematics, Bonn, Germany (gallauer@mpim-bonn.mpg.de)}\footnotewomarker{\textit{Keywords:} \kw. \textit{Mathematics Subject Classification (2020)}: 14F20, 14F08, 14F25, 14F42, 18G80, 18M05}\thanks{The author was supported by a Titchmarsh Fellowship of the University of Oxford.}
}
\date{}
\maketitle{}
\begin{abstract}
\noindent We develop a `universal' support theory for derived categories of constructible (analytic or \'etale) sheaves, holonomic $\mathcal{D}$-modules, mixed Hodge modules and others.
As applications we classify such objects up to the tensor triangulated structure and discuss the question of monoidal topological reconstruction of algebraic varieties.
\setcounter{tocdepth}{1}
\end{abstract}
\medskip

\section{Introduction}
\label{sec:introduction}

Let $X$ be an algebraic variety
and let $M$ be a constructible or perverse sheaf (in the analytic or \'etale topology), a holonomic $\cD$-module, or a mixed Hodge module on~$X$ (whenever these make sense).
The set $\Supp(M)$ of points $x\in X$ where $x^*M$ does not vanish is a constructible subset of~$X$.
In this article we show that $\Supp(M)$ is a fundamental homological invariant of~$M$.
Namely, letting $\CS(X)$ denote one of the corresponding `derived' categories $\Dbc(X)$, $\Dbh(\cD_X)$, $\Dbmhm{X}$ (whenever these make sense), we show
(cf.~\Cref{sta:main}):
\begin{thrm}
\label{thrm:intro-spc}
The assignment $M\mapsto\Supp(M)$ is the universal support datum on~$\CS(X)$.
In other words, it induces a homeomorphism
\[
\Spc(\CS(X))\cong\cons{X}.
\]
\end{thrm}
Here, the left-hand side denotes the spectrum of the tensor triangulated category $\CS(X)$ as defined by Balmer~\cite{balmer:spectrum}, and the right-hand side denotes the set of points of the scheme~$X$ with the constructible topology.
We point out the importance for this result that the category $\CS(X)$ is $\QQ$-linear (in contrast to positive characteristic).
We refer to \Cref{sec:examples} for our conventions regarding each of the theories mentioned as well as for a discussion of further examples.

\medskip
One may view this result from at least two different angles.
On the one hand, it says that the support \textbf{classifies} objects in these theories up to the tensor triangulated structure.
That is, two objects in $\CS(X)$ can be built out of each other using extensions, shifts, direct summands, and tensor products with arbitrary objects, if and only if their total cohomologies have the same support.
Equivalently, the support sets up a bijection between thick tensor ideals in $\CS(X)$ and ind-constructible subsets of~$X$.
Thereby the support is seen to play the same role for these theories as the chromatic level for stable homotopy theory~\cite{MR960945} and the $\pi$-support for the representation theory of finite group schemes~\cite{MR3718455}.

\medskip
On the other hand, one may ask how much information is lost by passing from~$X$ to, say, the abelian category of constructible sheaves on~$X$.
This is investigated in~\cite{kollar-lieblich-olsson-sawin:topological-reconstruction} where Koll\'ar--Lieblich--Olsson--Sawin prove a topological version of Gabriel's \textbf{reconstruction} theorem for a certain class of varieties (proper normal varieties of dimension at least two over an uncountable algebraically closed field of characteristic zero).
It is natural to wonder about monoidal and derived analogues and the authors in \loccit specifically suggest an approach using the spectrum of tensor triangulated categories.
A precursor is Thomason's result~\cite{MR1436741} which translates to the identity $\Spc(\Dperf(X))\cong \zar{X}$ for every quasi-compact quasi-separated scheme~$X$~\cite[Theorem~8.5]{MR2280286}.
We will refine \Cref{thrm:intro-spc} to obtain the following monoidal topological analogue (\Cref{sta:zariski-reconstruction}).
\begin{thrm}
\label{thrm:intro-spcs}%
Let $\CS(X)$ be as above.
Then the underlying Zariski topological space of~$X$ is completely determined by the tensor triangulated category~$\CS(X)$.
More precisely we have
\[
\Spcs(\CS(X))\cong \zar{X}.
\]
\end{thrm}
Here, the left-hand side denotes the \emph{smashing spectrum} of~$\CS(X)$, that is, the spectrum of the lattice of smashing tensor ideals in $\CS(X)$.
(We refer the reader to \Cref{rmk:spcs-informal} for an informal explanation why the smashing spectrum should play a role.)
The first, less precise statement in \Cref{thrm:intro-spcs} can also be proven along the lines of~\cite[\S\,5.4]{kollar-lieblich-olsson-sawin:topological-reconstruction}.
In any case, together with~\cite[Theorem~5.1.2]{kollar-lieblich-olsson-sawin:topological-reconstruction} we deduce that for proper normal varieties of dimension at least two over an uncountable algebraically closed field of characteristic zero, the tensor triangulated category $\CS(X)$ completely determines the scheme~$X$ (\Cref{sta:variety-reconstruction}).
The question as to whether that remains true without the tensor structure---in the spirit of Bondal--Orlov's reconstruction theorem~\cite{MR1818984}---would seem natural.
Similarly, it would be interesting to study the group of autoequivalences of~$\CS(X)$.

\medskip
Of course, if a phenomenon is observed in seemingly all cohomology theories it is natural to wonder about a \textbf{motivic} explanation.
Let us then denote by $\DMc(X)$ `the' category of constructible motivic sheaves on~$X$ with coefficients in a characteristic zero field, for example Beilinson motives~\cite{cisinski-deglise:DM} or Ayoub's \'etale motives~\cite{ayoub:etale-realization}.
If $X=\Spec(\kk)$ is the spectrum of a field, we had already established in~\cite[Theorem~3.8]{gallauer:tannaka-dmet} the analogue of \Cref{thrm:intro-spc} assuming `all conjectures on motives over~$\kk$'.
Here we show that this special case implies the result for arbitrary~$X$ and obtain the following result in \emph{motivic tensor-triangular geometry}.
\begin{thrm}
\label{thrm:spc-dm}
Let $X/\kk$ be a quasi-projective variety and assume that $\Spc(\DMc(\kk(x)))=\ast$ for each $x\in X$.
Then there are canonical homeomorphisms
\begin{align*}
\Spc(\DMc(X))\cong \cons{X},&&\Spcs(\DMc(X))\cong\zar{X}.
\end{align*}
\end{thrm}

\bigskip
The article is organized as follows.
In \Cref{sec:CS} we define \emph{generically simple constructible systems} which abstract the relevant features of the association $X\mapsto \CS(X)$.
The theories mentioned in \Cref{thrm:intro-spc,thrm:spc-dm} are shown to be examples in \Cref{sec:examples}, which involves establishing that $\CS(X)$ is a simple tensor triangulated category \emph{generically}.
This uses crucially (a step in) Beilinson's argument on the derived category of perverse sheaves~\cite{beilinson87}, as well as~\cite{gallauer:tannaka-dmet} which shows that the derived category of a Tannakian category in characteristic zero is simple.
In \Cref{sec:support} (resp.\ \Cref{sec:toprec}) we prove \Cref{thrm:intro-spc} (resp.\ \Cref{thrm:intro-spcs}) for (suitable) generically simple constructible systems, including the one of \Cref{thrm:spc-dm}.

\paragraph*{Acknowledgments}

Many thanks to Paul Balmer and Burt Totaro for precious feedback on a draft of this article.
A referee report led to many corrections and significant improvements in the text.
We are very grateful to the author of that report.

\section{Constructible systems}
\label{sec:CS}
In order to deal with constructible sheaves, holonomic $\cD$-modules, mixed Hodge modules etc.\ all simultaneously, it will be convenient to introduce a `minimal' set of axioms such theories should satisfy.
This is done in \Cref{sec:CS-defn}.
In \Cref{sec:new-from-old}, we provide some useful criteria to verify these axioms.
\subsection{Definition and basic properties}
\label{sec:CS-defn}

\begin{notn}
All schemes are assumed to be noetherian and reduced\footnote{Reducedness is entirely unnecessary and for convenience only; see \Cref{rmk:alternative-definition}}.
We fix a base scheme $\base$ and a full subcategory $\V$ of the category of $\base$-schemes satisfying the following property: If $V\into X$ is an immersion of $\base$-schemes and $X$ belongs to $\V$ then so does $V$.
\end{notn}

\begin{exa}
The structure map $X\to \base$ will play no role in this section and is there only for flexibility.
Indeed, some of the theories considered in \Cref{sec:examples} are not defined on all schemes.
The main examples for us are the following.
\begin{enumerate}[(a)]
\item $\V=\Var$, the category of \emph{$\kk$-varieties}, that is, separated, finite type schemes over a field~$\kk$.
\item $\V=\Sub[\base]$, the poset spanned by subschemes of $\base$.
\end{enumerate}
\end{exa}

Let $\ttCat$ denote the $(2,1)$-category of essentially small tensor triangulated categories (or, tt-categories for short) together with exact tensor functors (or, tt-functors for short) and natural isomorphisms of such.

\begin{defn}
\label{cs}
A \emph{constructible system} (on $\V$) is a pseudo-functor $\CS:\V\op\to\ttCat$ 
satisfying the axioms \bref{ax:loc} and \bref{ax:cons} below. 
\end{defn}

Given a morphism $f:X\to Y$ in $\V$, we denote the tt-functor $\CS(f)$ by $f^*$.
Sometimes, especially if $f$ is an immersion, we denote $f^*M$ by $M|_X$.
When $\V=\Sub[\base]$ we also speak of a constructible system \emph{on $\base$}.

\begin{notn}
\label{notn:C(x)}
Assume $X\in\V$ is integral, with generic point $x\in X$.
We denote by
\[
\CS(x):=\CS(x\in X):=\textup{2-colim} \CS(U)
\]
where the colimit in $\ttCat$ is over non-empty open subsets~$U\subseteq X$.\footnote{For a brief discussion of 2-colimits in $\ttCat$ see~\cite[Remark~8.3]{gallauer:tt-fmod}.}

If $X$ is not assumed integral, and $x\in X$ is an arbitrary point, with closure $\cl(x)\subseteq X$, we let $\CS(x):=\CS(x\in \cl(x))$.
Alternatively, it is the 2-colimit of $\CS(V)$ where $V$ runs through locally closed neighborhoods of~$x$ in~$X$.
We denote by $\rho_x:\CS(X)\to\CS(x)$ the canonical functor.
\end{notn}

Completing \Cref{cs} we impose the following axioms on $\CS$:
\begin{itemize}
\item \namedlabel{ax:loc}{(Loc)}\kern-.2em\footnote{standing for `localization'} For every open-closed decomposition $U\xinto{j}X\xfrom{i}Z$ ($U$ open, $Z$ the closed complement), the functors
  \[
    \CS(U)\xleftarrow{j^*}\CS(X)\xrightarrow{i^*}\CS(Z)
  \]
  define a recollement (see \Cref{recollement} below).
  \item \namedlabel{ax:cons}{(Lisse)}
  There exists, for each regular $U\in\V$, a full tt-subcategory $\CSls(U)\subseteq\CS(U)$ (of so-called \emph{lisse} objects) such that:
  \begin{enumerate}[(1)]
  \item  For every $X\in\V$ and $M\in\CS(X)$ there exists a dense regular open subset $U\subseteq X$ such that $M|_U\in\CSls(U)$.
  \item \label{lisse:restriction} For every immersion $f$ between regular schemes, the functor $f^*$ preserves lisse objects.
  \item \label{lisse:conservative} For regular, connected $X\in\V$ and $x\in X$, the functor $\rho_x|_{\CSls(X)}:\CSls(X)\to\CS(x)$ is conservative.
\end{enumerate}
Given a constructible system $\CS$ we will often implicitly assume that a choice of $\CSls$ as in \bref{ax:cons} has been made.

\end{itemize}

\begin{rmk}
\label{rmk:terminology}%
These axioms are modeled on the behavior of categories of constructible sheaves, the lisse objects corresponding to local systems.
(In examples of interest, lisse objects are often the rigid objects.)
On the other hand, the examples of constructible systems we have in mind (\Cref{sec:examples}) all arise from (subfunctors of) fully-fledged six-functor formalisms, or, \emph{coefficient systems}~\cite{drew:MHM,drew-gallauer:usf,2021arXiv211210456G}.
The term ``constructible system'' is an attempt at capturing these two aspects.
\end{rmk}

\begin{rmk}
\label{rmk:alternative-definition}%
By \bref{ax:loc}, there is a unique way to extend a constructible system $\CS$ to non-reduced schemes, namely by setting $\CS(X):=\CS(X_{\textup{red}})$.
This extended constructible system satisfies the analogous axioms, and the arguments in the article go through with minor modifications.
In particular, both \Cref{sta:main,sta:zariski-reconstruction} remain true.
\end{rmk}

\begin{notn}
  \label{recollement}%
  Let ${\cal T}_U\xleftarrow{j^*}{\cal T}\xrightarrow{i^*}{\cal T}_Z$ be triangulated categories and exact functors between them.
  This is called a \emph{recollement} if
  \begin{enumerate}[(1)]
  \item $j^*$ admits a fully faithful right adjoint $j_*$;
  \item $i^*$ admits a fully faithful right adjoint $i_*$ whose essential image is the kernel of $j^*$.
  \end{enumerate}

\end{notn}

\begin{rmk}
  \label{recollement-bbd}
  By \cite[Proposition~4.13.1]{krause:localization-theory-tricat}, a recollement as in \Cref{recollement} induces a recollement in the sense of~\cite[\S\,1.4.3]{MR751966}:
  two exact functors ${\cal T}_Z\xrightarrow{i_*}{\cal T}\xrightarrow{j^*}{\cal T}_U$ satisfying the following properties:
  \begin{enumerate}[(1)]
  \item $i_*$ admits left and right adjoints $i^*$, $i^!$, respectively.
  \item $j^*$ admits left and right adjoints $j_!$, $j_*$, respectively.
  \item $j^*i_*=0$.
  \item $i_*, j_*, j_!$ are fully faithful.
  \item There are functorial triangles
    \[
      j_!j^*\to \id\to i_*i^*\to^+,\qquad i_*i^!\to\id\to j_*j^*\to^+
    \]
  \end{enumerate}
From now on we will identify these two equivalent sets of data.
\end{rmk}

\begin{notn}
\label{notn:extension-by-zero}%
Let $f:V\into X$ be an immersion in $\V$.
We associate a functor $f_!:\CS(V)\to\CS(X)$, called \emph{extension by zero} as follows.
Factor $f=ij$ as an open immersion~$j$ followed by a closed immersion~$i$ and define $f_!=i_*j_!$.
It follows from \bref{ax:loc} that this is independent of the choice of factorization, up to (canonical) isomorphism.
(And the `same' functor is obtained  by factoring $f$ as a closed immersion followed by an open immersion.)
We often denote the composite $f_!f^*$ by $(-)_V$.
\end{notn}

The following \emph{projection formul\ae} will be useful in the sequel.
Note that the morphisms `dual' to~(\ref{eq:proj-formula}) involving the pairs of functors $(j^*, j_*)$ and $(i_*,i^!)$ are in general not invertible.
\begin{lem}
\label{sta:projection-formulae}%
Let $\CS$ be a constructible system and let $j,i$ as in \bref{ax:loc}.
For $K\in\CS(U)$, $L\in \CS(X)$, $M\in\CS(Z)$ the following canonical maps are invertible:
\begin{equation}
\label{eq:proj-formula}%
j_!(M\otimes j^*N)\xto{\sim} j_!M\otimes N,\qquad i_*M\otimes N\xto{\sim} i_*(M\otimes i^*N)
\end{equation}
\end{lem}
\begin{proof}
  We note that the pair $(j^*,i^*)$ is conservative, as follows from the first triangle in \Cref{recollement-bbd}.
  It is straight-forward to check that upon applying these two functors, the maps in the statement become isomorphisms thus the claim.
\end{proof}

\begin{rmk}
\label{rmk:C(x)-lisse}%
It follows from \bref{ax:cons} that $\CS(x)=\CSls(x):=\textup{2-colim}\CSls(U)$ over the regular locally closed neighborhoods of~$x$.
Indeed, the natural comparison functor $\CSls(x)\to\CS(x)$ is fully faithful, and \bref{ax:cons} ensures that it is essentially surjective too.
\end{rmk}

In this article we will not be able to say much about constructible systems in general.
Instead we will restrict to `generically simple' ones in the following sense.
\begin{defn}
  \label{scs}%
  Let $\CS$ be a constructible system.
  It is called \emph{generically simple} if it satisfies the following additional axiom.
\begin{itemize}
\item \namedlabel{ax:simple}{(GenS)} For every $x\in X\in\V$, the tt-category $\CS(x)=\CSls(x)$ is simple.
\end{itemize}
\end{defn}
Here we call a tt-category \emph{simple} if it has exactly two thick tensor ideals.
These are necessarily the zero ideal and the whole category.\footnote{If the tt-category is rigid, this is equivalent to the definition used in~\cite[p.\,119]{gallauer:tannaka-dmet}.
  In the examples of constructible systems of interest to us, the category in question will always be rigid.}
(Recall that a thick tensor ideal (or, tt-ideal for short) is a full triangulated subcategory stable under direct summands and tensoring with arbitrary objects.)

\begin{defn}
\label{defn:primes}%
Fix a generically simple constructible system $\CS$ and let $x\in X\in\V$.
We denote by $\pr_x$ the kernel of $\rho_x:\CS(X)\to\CS(x)$ (see \Cref{notn:C(x)}).
By \bref{ax:simple}, $\pr_x$ is a prime tt-ideal (that is, $\one\notin\pr_x$, and if $M\otimes N\in\pr_x$ then $M\in\pr_x$ or $N\in\pr_x$).
\end{defn}

\begin{lem}
\label{sta:primes-functoriality}%
With the assumptions of \Cref{defn:primes} let $f:X\to Y$ be a morphism in $\V$.
Then we have for all $x\in X$:
\[
\pr_{f(x)}=(f^*)^{-1}\pr_x
\]
\end{lem}
\begin{proof}
This follows from the commutativity of the square
\[
\begin{tikzcd}
C(Y)
\ar[r, "f^*"]
\ar[d, "\rho_{f(x)}"]
&
C(X)
\ar[d, "\rho_x"]
\\
C(f(x))
\ar[r, "f^*"]
&
C(x)
\end{tikzcd}
\]
together with the fact that a tt-functor between simple tt-categories is conservative.
\end{proof}

\subsection{New from old}
\label{sec:new-from-old}

We now discuss a few criteria that will be useful in checking the axioms of a (generically simple) constructible system.

\begin{lem}
\label{sta:lift-cs-conservative}%
Let $\CS$ be a constructible system and let $\omega:\CSp\to\CS$ be a pseudo-natural transformation of pseudo-functors $\V\op\to\ttCat$ such that:
\begin{enumerate}[(1)]
\item for each $X\in\V$, the functor $\omega_X:\CSp(X)\to\CS(X)$ is conservative;
\item for each immersion $f:V\into X$, the functor $f^*:\CSp(X)\to\CSp(V)$ admits a right adjoint $f_*$, and the transformation  $\omega_X f_*\to f_*\omega_V$ is invertible.
\end{enumerate}
Then $\CSp$ is a constructible system.
\end{lem}
\begin{proof}
With $\CSpls$ the preimage of $\CSls$ under~$\omega$, the axiom \bref{ax:cons} is clear.
For \bref{ax:loc}, let $j:U\into X$ and $i:Z\into X$ be as in the axiom.
Note that fully faithfulness of $j_*$ (resp.\ $i_*$) is equivalent to the counit $j^*j_*\to \id$ (resp.\ $i^*i_*\to\id$) being invertible.
These are mapped to the corresponding counits in $\CS$, by assumption, and are therefore invertible by conservativity of~$\omega$.
Similarly, $\omega$ allows us to conclude that $j^*i_*=0$ so that the image of $i_*$ is contained in~$\ker(j^*)$.
Conversely, let $M\in\ker(j^*)$ and consider the unit $M\to i_*i^*M$.
Since $\omega$ inverts this morphism it is invertible and we win.
\end{proof}

\begin{cor}
\label{sta:cs-sub}%
Let $\CS$ be a constructible system and let $\CSp$ be a sub-pseudo-functor stable under $f_*$ for each immersion~$f$.
\begin{enumerate}[(a)]
\item Then $\CSp$ is a constructible system.
\item Assume that $\CSp(x)$ is rigid for each $x\in X\in\V$. If $\CS$ is generically simple then so is $\CSp$.
\end{enumerate}
\end{cor}
\begin{proof}
The first statement follows directly from \Cref{sta:lift-cs-conservative}.
For the second statement let $x\in X\in\V$.
Since a filtered colimit of faithful functors is faithful we see that the canonical tt-functor $\CSp(x)\to \CS(x)$ is faithful.
The claim now follows from~\cite[Corollary~1.8]{balmer:surjectivity}.
\end{proof}

The following criterion for generic simplicity leverages an observation from~\cite{gallauer:tannaka-dmet} together with Deligne's internal characterization of Tannakian categories.
For both of these the characteristic zero assumption is crucial.
Here we denote by $\tCat$ the (2,1)-category of essentially small symmetric monoidal categories.
\begin{prop}
\label{sta:cs-gen-simple-criterion}%
Let $\EE$ be a characteristic zero field.
Let $\cA_\bullet:I\to\tCat$ be a filtered diagram of $\EE$-linear Tannakian categories\footnote{By convention, this involves the condition $\EE=\mathrm{End}(\one)$.} and $\EE$-linear exact transition functors.
Denote by $\cA=\textup{2-colim}_i\cA_i$ the colimit.
Then the tt-category $\Db(\cA)$ is simple.
\end{prop}
\begin{proof}
The colimit $\cA$ is an $\EE$-linear, rigid tensor abelian category, and it follows from~\cite[Th\'eor\`eme~7.1]{MR1106898} that it is Tannakian too.
We proved in~\cite[Theorem~2.4]{gallauer:tannaka-dmet} that the bounded derived category of a Tannakian category in characteristic zero is simple.
\end{proof}

\section{Examples}
\label{sec:examples}

In this section we verify that the theories mentioned in the introduction satisfy the axioms of a generically simple constructible system.
We also discuss additional theories obtained by restriction to subcategories, as well as the theory of motivic sheaves which satisfies (some of) the axioms only conjecturally.

\subsection{Constructible analytic sheaves}
\label{sec:dbc-analytic}
Let $\kk\subseteq\CC$ be a field with a fixed embedding into the complex numbers and fix a characteristic zero field~$\EE$.
For a variety $X\in\Var$ we denote by $X(\CC)$ the set of $\CC$-points with the usual analytic topology, and by
$\Sh(X(\CC);\EE)$ the category of sheaves of $\EE$-vector spaces on this topological space.
Let $\LS(X(\CC);\EE)$ denote the full subcategory of $\EE$-local systems and $\Cons(X;\EE)$ the full subcategory of (algebraically) constructible sheaves.
Recall the latter are those sheaves~$F$ for which there exists a finite stratification $X=\amalg_iX_i$ into locally closed subsets (defined over~$k$) with $F|_{X_i(\CC)}$ a local system.
The reader can find more details for example in~\cite[\S\,4.1]{MR2050072} (for $\kk=\CC$ but this condition is unnecessary).
Both $\LS(X(\CC);\EE)$ and $\Cons(X;\EE)$ are Serre tensor subcategories of $\Sh(X(\CC);\EE)$ and hence restricting to complexes whose total cohomology belongs to them defines two sub-tt-categories
\[
\Dbls(X(\CC);\EE)\subseteq \Dbc(X;\EE)\subseteq\Der(\Sh(X(\CC);\EE))
\]
of the derived category of all sheaves.

\begin{prop}
\label{sta:Dbc-gscs}%
The pseudo-functor $\Dbc(-;\EE)$ is a constructible system on~$\Var$.
For $k$ algebraically closed it is moreover generically simple.
\end{prop}
\begin{proof}
The assignment $X\mapsto\Dbc(X;\EE)$ underlies a six-functor formalism, of which \bref{ax:loc} is only one property.
(If $\kk=\CC$, the necessary constructions and properties can be found for example in~\cite{MR2050072} and references therein, particulary in \cite[\S\,4.1]{MR2050072}.
For general $\kk$ see~\cite{ayoub-anabel}.)
With $\Dbls$ as the lisse objects, the first two conditions of \bref{ax:cons} are straightforward.
For the last condition, that is, to prove that $\rho_x|_{\CSls(X)}$ is conservative when $X$ is (regular and) connected and $x\in X$, we reduce to the following claim:
If $F\in\LS(X(\CC);\EE)$ vanishes on $V(\CC)$ for some locally closed subset $x\in V\subseteq X$ then $F=0$.
As a local system has constant rank on a connected topological space it suffices to prove that $V(\CC)$ meets all connected components of $X(\CC)$.
Denote by $X_\CC$ (resp.\ $V_\CC$) the base change of $X$ (resp.\ $V$) to~$\CC$.
The canonical map $\pi_0(X_\CC)\to\pi_0(X(\CC))$ is a bijection (and similarly for~$V$) so we reduce to proving that $V_\CC$ meets all connected components of $X_\CC$.
Let $T\subseteq X_\CC$ be such a connected component.
By~\cite[\href{https://stacks.math.columbia.edu/tag/04PZ}{Tag 04PZ}]{stacks-project}, the image of $T$ in~$X$ is a connected component hence must equal~$X$.
In particular, $T$ and $V_\CC$ intersect.

For \bref{ax:simple} let $X$ be an irreducible variety with generic point~$x$.
For each open $x\in U\subseteq X$, the canonical tt-functor
\[
\Db(\Cons(U;\EE))\to\Dbc(U;\EE)
\]
is an equivalence, as Nori proves~\cite[Theorem~3(b)]{MR1940678}.
Passing to the 2-colimit yields a tt-equivalence $\Dbc(x;\EE)\simeq \Db(\Cons(x;\EE))$ where $\Cons(x;\EE)$ is defined as the analogous 2-colimit of the $\Cons(U;\EE)$.
(See~\cite[Lemma~2.6]{gallauer:tannaka-dmet} for commuting $\Db$ and filtered colimits if necessary.)
By definition of constructible sheaves, we also have $\Cons(x;\EE)=\LS(x;\EE)$.
But the category of local systems $\LS(U(\CC);\EE)$ is $\EE$-linear neutral Tannakian if $\kk$ is algebraically closed.
(If $\kk$ is not algebraically closed $U(\CC)$ can have multiple connected components.)
We conclude with \Cref{sta:cs-gen-simple-criterion} that $\Dbc(x;\EE)$ is simple.
\end{proof}

\begin{rmk}
\label{rmk:LSx}%
For each $U$ as in the proof, a fiber functor for $\LS(U(\CC);\EE)$ is given by the stalk at any point $u$ of $U(\CC)$, in which case we obtain an identification with the category $\Rep(\pi_1(U(\CC),u);\EE)$ of $\EE$-linear representations of the fundamental group.
We argued in the proof of \Cref{sta:cs-gen-simple-criterion} that the filtered colimit $\LS(x;\EE)$ of these Tannakian categories remains Tannakian (although not necessarily neutral) using Deligne's characterization of Tannakian categories in characteristic zero.
Alternatively, a fiber functor~$\omega$ can be described as follows.

Assume $X$ is integral and choose an ultrafilter $\cU$ on $X(\CC)$ which contains $U(\CC)$ for all $U$ appearing in the colimit.
(This exists because every $U$ is dense.)
Let $\EE(x)=\EE^{X(\CC)}/\cU$ be the associated ultrapower, that is, the quotient of the ring $\EE^{X(\CC)}$ by the equivalence relation whereby two families become equivalent iff they agree on an element of $\cU$.
Then $\EE(x)$ is a field extension of $\EE$.
Given a local system $M\in\LS(U(\CC);\EE)$ define $\omega(M)$ to be the equivalence class of the family $(M_u)$ where $M_u$ is the stalk at $u$ if $u\in U(\CC)$ and $0$ else.
\textsl{Apriori} this functor takes values in the ultrapower $\modf(\EE)^{X(\CC)}/\cU$ but clearly it actually lands in the abelian tensor subcategory of objects with constant finite dimension almost everywhere.
One may identify this subcategory with $\modf(\EE(x))$, the category of finite dimensional $\EE(x)$-vector spaces.
This defines a fiber functor $\omega:\LS(x;\EE)\to\modf(\EE(x))$.
\end{rmk}

In fact, the constructible system $\Dbc(-;\EE)$ is generically simple in an even stronger sense.
The following result will not be used in the sequel but seems interesting in its own right.
It is implicit in Beilinson's proof of~\cite[Lemma~2.1.1]{beilinson87}.
\begin{prop}
\label{sta:dbls-generic}%
Let $X\in\Var[\CC]$ and $\EE$ a characteristic zero field.
There exists an open dense $U\subseteq X$ such that the canonical tt-functor
\[
\Db(\LS(U(\CC);\EE))\xto{\sim}\Dbls(U(\CC);\EE)
\]
is an equivalence.
\end{prop}
\begin{proof}
We may clearly assume $X$ is integral.
In~\cite[Lemma~2.1.1]{beilinson87} it is proved (as a special case) that for the generic point $x\in X$, the canonical functor $\Db(\LS(x;\EE))\to\Dbls(x;\EE)$ is an equivalence.
The proof proceeds by induction on the dimension of~$X$.
In the induction step one starts with a local system $F$ (in Beilinson's notation, $F=M^*\otimes N$) on $X$ and produces a smooth affine morphism $\pi:U\to Z$ with 1-dimensional fibers, where $U\subseteq X$ is a non-empty open subset, and $Z$ is regular, such that the higher direct images $\dR^q\pi_*(F|_U)$ are local systems on~$Z$.

To obtain the statement in the proposition it is sufficient to choose $\pi$ independently of~$F$.
This is indeed possible, by~\cite[Corollaire~5.1]{MR481096}.
For example, starting with an elementary fibration $\pi':U'\to Z'$ in the sense of Artin~\cite[\S\,XI]{MR0354654},
this result guarantees a (regular) open dense $Z\subseteq Z'$ such that the restriction $\pi=\pi'|_U:U=\pi^{-1}(Z)\to Z$ is a locally trivial fibration in the analytic topology.
In particular, for \emph{every} $F$ as above, the higher direct images $\dR^q\pi_*(F|_U)$ are local systems on~$Z$.
The rest of the proof of~\cite[Lemma~2.1.1]{beilinson87} can be copied verbatim.
\end{proof}

\begin{rmk}
\begin{enumerate}[(a)]
\item We do not know whether \Cref{sta:dbls-generic} holds for $\ell$-adic lisse sheaves (discussed below in \Cref{sec:dbc-adic}).
\item The statement of \Cref{sta:dbls-generic} as well as the invocation of elementary fibrations in the proof suggest a relation with $K(\pi,1)$-spaces (`Artin neighborhoods'~\cite[\S\,XI]{MR0354654}).
However, we do not know whether every $K(\pi,1)$-space~$U$ satisfies the conclusion.
Indeed, saying that $U$ is a $K(\pi,1)$-space amounts to the statement that extensions between local systems may be computed equivalently in the category of sheaves on $U(\CC)$ or in the category of $\pi_1(U(\CC))$-modules.
However, we need to know that they may be computed in the category of local systems on $U(\CC)$ (that is, in the category of \emph{finite dimensional} $\pi_1(U(\CC))$-modules).
\end{enumerate}
\end{rmk}

For $X\in\Var[\CC]$ denote by $\Perv(X;\EE)\subseteq\Dbc(X;\EE)$ the abelian category of perverse sheaves on~$X$, and by
$\Pervgm(X;\EE)$ the full subcategory of $\Perv(X;\EE)$ spanned by objects whose simple subquotients are all of geometric origin~\cite[\S\,6.2.4]{MR751966}.
By construction, this is a weak Serre tensor subcategory.\footnote{A weak Serre subcategory of an abelian category $\cC$ is a non-empty full subcategory $\cC'$ such that an exact sequence $A_1\to A_2\to B\to A_3\to A_4$ in $\cC$ with $A_i\in\cC'$ implies $B\in\cC'$.}
Let $\Dbgm(X;\EE)$ be the subcategory of $\Dbc(X;\EE)$ spanned by objects whose perverse cohomologies belong to $\Pervgm(X;\EE)$.
These are the \emph{constructible analytic sheaves of geometric origin}.
In addition to being a tt-subcategory, it is stable under the usual functoriality in~$X$, by construction.
\begin{prop}
\label{sta:Dbgm-gscs}%
The pseudo-functor $\Dbgm(-;\EE)$ is a generically simple constructible system on~$\Var[\CC]$.
\end{prop}
\begin{proof}
This follows from \Cref{sta:cs-sub} and \Cref{sta:Dbc-gscs}.
\end{proof}

\begin{rmk}
\label{rmk:quasi-unipotent}%
For an alternative definition of constructible sheaves of geometric origin and a proof that they are stable under the required functoriality see~\cite{ayoub-anabel} which also treats the general case $\kk\subseteq\CC$.

Between the two generically simple constructible systems $\Dbgm\subseteq\Dbc$, another example should be generated by quasi-unipotent constructible sheaves~\cite{MR656052}.
\end{rmk}

\subsection{Constructible \'etale sheaves}
\label{sec:dbc-adic}

Let $\kk$ be a field and $\ell$ a prime number invertible in~$\kk$.
Fix a finite extension $\EE$ of $\QQ_\ell$, or $\EE=\bar{\QQ}_\ell$. 
We will not recall the definition of the `derived' tt-category~$\Dbc(X;\EE)$ of constructible adic sheaves in the \'etale topology on varieties $X/k$.
This is sketched in~\cite[\S\,1.1]{MR601520} for certain $\kk$, and treated carefully (and generalized) in~\cite{MR1106899}.
Equivalently, one may use the pro-\'etale topology to define these categories~\cite{MR3379634}.
There is a bounded t-structure whose heart is the category $\Cons(X;\EE)$ of constructible $\EE$-sheaves on~$X$ in the sense of~\cite[Expos\'e~VI, 1.4.3]{sga5}.
Let $\LS(X;\EE)\subseteq\Cons(X;\EE)$ denote the full subcategory of lisse $\EE$-sheaves.
As in the analytic context this gives rise to a tt-subcategory
$\Dbls(X_\et;\EE)\subseteq\Dbc(X;\EE)$ and we claim:
\begin{prop}
\label{sta:Dbcadic-gscs}%
The pseudo-functor $\Dbc(-;\EE)$ is a generically simple constructible system on~$\Var$.
\end{prop}
\begin{proof}
The proof that $\Dbc(-;\EE)$ is a constructible system is entirely analogous to \Cref{sta:Dbc-gscs}.
We start by noting that $\Dbc(-;\EE)$ admits a six-functor formalism and therefore satisfies \bref{ax:loc}, see \cite[Theorem~6.3.(iv)]{MR1106899} for a precise statement.
With $\Dbls$ as the lisse objects the first condition in \bref{ax:cons} holds by construction.
The second condition holds (for any morphism~$f$, not necessarily an immersion) because $f^*$ preserves constant sheaves.
For the last condition, using the t-structure we reduce to showing that the functor
\[
\LS(X_\et;\EE)\to\LS(x;\EE):=\textup{2-colim}\LS(V_{\et};\EE)
\]
is conservative when $X$ is connected, and where $V\subseteq X$ runs through the locally closed subsets containing~$x$.
But if $M\in\LS(X_\et;\EE)$ vanishes on a locally closed subset~$V$ and $\bar{x}$ is a geometric point of~$X$ with image~$x$ then $M_{\bar{x}}=0$ which implies $M=0$.
Indeed, $(-)_{\bar{x}}$ is a fiber functor which exhibits $\LS(X_\et;\EE)$ as a neutral Tannakian category~\cite[Expos\'e~VI, \S\,1.4.3]{sga5}.

We now turn to the last axiom \bref{ax:simple}.
If $\kk$ is algebraically closed, one can prove it literally as in \Cref{sta:Dbc-gscs}, using~\cite{barrett:dbc} instead of~\cite{MR1940678}.
In the general case, let $x\in X\in\V$.
By the argument of~\cite[Lemma~2.1.1]{beilinson87}, the canonical tt-functor $\Db(\Cons(x;\EE))\to\Dbc(x;\EE)$ is an equivalence.
Indeed, while Beilinson works with the \emph{perverse} t-structure in \loccit, the argument clearly goes through with respect to the standard t-structure as well, and for non-closed fields~$\kk$.
But $\Cons(x;\EE)=\LS(x;\EE)$ and this is a filtered colimit of neutral $\EE$-linear Tannakian categories $\LS(V_\et;\EE)$, as already explained in the first paragraph of the proof.
The axiom \bref{ax:simple} therefore follows from \Cref{sta:cs-gen-simple-criterion}.
\end{proof}

\begin{rmk}
The category $\LS(x;\EE):=\textup{2-colim}\LS(V_\et;\EE)$ admits a fiber functor analogously to the analytic case in \Cref{rmk:LSx}.
\end{rmk}

\begin{rmk}
\label{rmk:mixed-sheaves}%
The subsystem generated by mixed sheaves~\cite[\S\,6]{MR601520} should form another example of a generically simple constructible system.
\end{rmk}

\subsection{Holonomic \texorpdfstring{$\mathcal{D}$}{D}-modules}

Let $k$ be a field of characteristic zero.
For a \emph{regular} $k$-variety~$X$ we denote by $\Dbh(\cD_X)$ the bounded derived category of holonomic $\cD_X$-modules~\cite[\S\,VI.1.13]{bgkhme:d-modules}.
It comes with a canonical conservative exact functor $\nu:\Dbh(\cD_X)\to\Dbqc(\cO_X)$.
Since the notation in the literature is somewhat inconsistent let us stress that the tensor product and inverse image functor along $f:X\to Y$ `induced' by the tensor product and inverse image of $\cO_X$-modules will here be denoted by $\otimes^!$ and $f^!$, respectively.\footnote{In~\cite[\S\,VI]{bgkhme:d-modules}, these are denoted by $\otimes_{\cO_X}^{\dL}$ and $f^!$, respectively.}
More precisely, we have for $M,N\in\Dbh(\cD_Y)$ the relations $\nu(M\otimes^!N)\cong \nu(M)\otimes^{\dL}\nu(N)[-\dim(Y)]$ and $\nu f^!(M)\cong \dL f^*(\nu(M))[\dim(X)-\dim(Y)]$.
We denote the Verdier dual tensor product and inverse image by $\otimes$ and $f^*$, respectively.\footnote{In~\cite[\S\,VI]{bgkhme:d-modules}, the latter is denoted by $f^+$.}
It is the latter two which we will use in defining the constructible system.
However, since the two tensor products are anti-equivalent, the conclusion in \Cref{sta:main} doesn't depend on this choice.
So to summarize, we have the usual adjunctions and relations, familiar from six-functor formalisms:
\begin{align*}
f_!\dashv f^!,&& f^*\dashv f_*,&& \mathbb{D}f^!\mathbb{D}=f^*,&&\mathbb{D}f_!\mathbb{D}=f_*,
\end{align*}
where $f^*$ is the tensor functor with respect to~$\otimes$.
\begin{prop}
\label{sta:dbh-gscs}%
Assume $X$ may be embedded in a regular variety.
The assignment $V\mapsto\Dbh(\cD_V)$ for regular $V\subseteq X$ extends to a generically simple constructible system on~$X$.\footnote{If the reader accepts that $\Dbh(\cD_{-})$ underlies a well-behaved six-functor formalism (e.g.\ satisfying the axioms of~\cite[D\'efinitions~1.4.1, 2.3.1, 2.3.50]{ayoub07-thesis} restricted to regular varieties) then there is a standard way of extending it to a well-behaved formalism on $\Var$.
In particular, the statement of the proposition should hold with $\Var$ instead of $\Sub$.
For lack of reference we proceed instead in an \textsl{ad-hoc} fashion.}
\end{prop}
\begin{proof}
Let $X\into Y$ be a closed embedding into a regular $k$-variety.
If we can prove the statement for $Y$ then we obtain the statement for $X$.
In other words, we may assume $X$ is regular to start with.
Set $\CS(X)=\Dbh(\cD_X)$.
Given $V\subseteq X$ a locally closed subset, choose an open $j:U\into X$ such that $V\subseteq U$ is a closed subset.
Temporarily denote by $\CS^{(U)}(V)$ the full subcategory of $C(X)$ spanned by objects isomorphic to $j_!M$ where $M|_{U\backslash V}=0$.
We claim that this is independent of the choice of~$U$.
Indeed, if $V$ is already closed in~$X$ (one easily reduces to this case) and $M|_{U\backslash V}=0$ then $(j_!M)|_{X\backslash V}\cong q_!(M|_{U\backslash V})=0$, where $q:U\backslash V\into X\backslash V$.
This shows that $\CS^{(U)}(V)\subseteq\CS^{(X)}(V)$.
Conversely, let $M\in\CS^{(X)}(V)$ so that $M|_{X\backslash V}=0$.
We need to show that $j_!j^*M\to M$ is invertible, or by Verdier duality, that $\DD M\to j_*j^*\DD M$ is invertible.
The fiber of this morphism is $\Gamma_{X\backslash U}(\DD M)$ and we know that $\nu\Gamma_{X\backslash U}(\DD M)\cong\Gamma_{X\backslash U}(\nu\DD M)=0$ since $(\DD M)|_{X\backslash V}=0$.
We conclude that $\CS^{(U)}(V)=\CS^{(X)}(V)$ and we may unambigously write $\CS(V)$ from now on.

Continuing with the same notation, we define a restriction functor $|_V:\CS(X)\to\CS(V)$ by sending $M$ to
\[
j_!\cone(r_!r^*j^*M\to j^*M)
\]
where $r:U\backslash V\into U$.
This is indeed functorial because there is, for $f:M\to N$, a \emph{unique} morphism $M|_V\to N|_V$ fitting into a morphism of the obvious distinguished triangles.\footnote{It suffices to show $\hom(j_!r_!r^*j^*M[1],N|_V)=0$, or by adjunction, $\hom(r^*j^*M[1],r^*j^*(N|_V))=0$ but $r^*j^*(N|_V)$ is the cone of the isomorphism $r^*r_!r^*j^*N\to r^*j^*N$ hence vanishes.}
Exploiting this same fact and using the octahedron axiom one may show that $|_V$ is canonically an exact functor.
Moreover, since $|_V$ restricts to the identity functor on $C(V)$ there is an essentially unique way of turning $C(V)$ into a tt-category so that $|_V$ is a tt-functor.
It is given by $M\otimes_{C(V)}N:=(M\otimes N)|_V$ with unit $\one|_V$.

For two different choices of $U$ there is a \emph{canonical} isomorphism between the two restriction functors (by the same argument as in the first paragraph) and from this one deduces that there is a pseudo-functor $\CS:\Sub[X]\op\to\ttCat$ sending $V$ to $\CS(V)$ and sending an inclusion of locally closed subsets $V\subseteq W$ to the restriction $|^W_V:=(|_V)|_{\CS(W)}:\CS(W)\to\CS(V)$.
This notation is not too abusive since for regular $V\subseteq X$, we have a canonical equivalence $\CS(V)\simeq\Dbh(\cD_V)$, by Kashiwara's lemma, which is compatible with the restriction functors.

We now verify the axioms of a constructible system.
For \bref{ax:loc}, let $U\xinto{j}V\xfrom{i}Z$ be an open-closed partition of $V$, where $V\subseteq X$ is a closed subset.
(Again, one easily reduces to this particular case of the axiom.)
In that case $X\backslash Z$ is an open and we denote by $j':X\backslash Z\into X$ the inclusion.
Then $U\subset X\backslash Z$ is a closed subset and the restriction $|^V_U:\CS(V)\to\CS(U)$ is given by $M\mapsto j'_!(j')^*M$ which has a right adjoint $N\mapsto j'_*(j')^*N$.
The restriction $|^V_Z:\CS(V)\to\CS(Z)$ is given by $M\mapsto \cone(j'_!(j')^*M\to M)$ which has a right adjoint given by the inclusion $\CS(Z)\into\CS(X)$.
The image is by construction the kernel of $|^V_U$.

For \bref{ax:cons}, let $V\subseteq X$ be a regular subvariety.
Under the canonical equivalence $\Dbh(\cD_V)\simeq\CS(V)$ we let $\CSls(V)$ be the full subcategory of $\CS(V)$ spanned by complexes whose cohomologies are integrable connections.
Equivalently, it is spanned by objects isomorphic to $j_!i_*M$ where $V\xinto{i}U\xinto{j}X$ is a factorization of a closed and an open immersion, and where $M\in\Dbh(\cD_V)$ is lisse in the sense above.
Note that a holonomic $\cD$-module $M$ is an integrable connection iff $\nu M$ is a vector bundle.
Also note that $M\in\Dbh(\cD_V)$ is lisse iff its Verdier dual $\mathbb{D}M$ is. 
It is well-known that every holonomic $\cD$-module is generically lisse thus the first condition in \bref{ax:cons}.
For the second one let $f:V\into W$ be an immersion between regular subvarieties of~$X$ and let $M\in\Dbh(\cD_W)$ be lisse.
Then $\nu \mathbb{D} f^*M\cong\nu f^!\mathbb{D}M\cong\dL f^*\nu\mathbb{D}M[d_f]$ (where $d_f=\dim(V)-\dim(W)$) and the claim follows from the fact that $\dL f^*$ preserves vector bundles.
The last condition follows, using the conservative functor $\nu$, from the fact that the support of a vector bundle is both open and closed.

To prove \bref{ax:simple} fix an integral subvariety~$V\subseteq X$ with generic point $x$.
We may clearly assume $V$ is regular.
For each open $x\in U\subseteq V$ we denote by $\hDmod{U}$ the category of holonomic $\cD_U$-modules and by $\hDmod{x}:=\textup{2-colim}_U\hDmod{U}$ the colimit.
By~\cite[Lemma~2.1.1]{beilinson87}, the canonical functor $F:\Db(\hDmod{x})\to\Dbh(\cD_x)$ is an equivalence of triangulated categories.
Moreover, $\hDmod{x}=\iC(x)$ is the colimit of the $\kk$-linear Tannakian categories of integrable connections $\iC(U)$ (with the usual tensor product over $\cO_U$).
By \Cref{sta:cs-gen-simple-criterion}, we see that the domain of~$F$, the tt-category $\Db(\iC(x))$, is simple.
However, $F$ is \textsl{apriori} not a tensor functor so some care needs to be taken to conclude.
Assume \textsl{ad absurdum} that $\Dbh(\cD_x)$ contains a non-trivial tt-ideal~$\cK'$, that is, $\cK'\neq 0,\Dbh(\cD_x)$.
Applying Verdier duality we see that also $\Dbh(\cD_x)$ with the tensor structure given by~$\otimes^!$ contains a non-trivial tt-ideal $\cK:=\mathbb{D}\cK'$.
We will obtain a contradiction by showing that the non-trivial thick subcategory $F^{-1}(\cK)\subseteq\Db(\iC(x))$ is a tt-ideal.
For this it suffices to show that $F^{-1}(\cK)$ is closed under tensoring with integrable connections since these generate $\Db(\iC(x))$.
Given $M\in F^{-1}(\cK)$ and $N\in\iC(x)$ choose $U\subseteq V$ so that $M\in\Db(\iC(U))$ and $N\in\iC(U)$.
Then we have $F(M)\otimes^! F(N)\in\cK$ and the latter is represented by
\[
F(M)\otimes^{\dL}_{\cO_U}F(N)[-d]\cong F(M\otimes N)[-d].
\]
It follows that $F(M\otimes N)\in\cK$ and we arrive at the required contradiction.
\end{proof}

\begin{rmk}
\label{rmk:mic(x)}
The category $\iC(x)=\textup{2-colim}\iC(U)$ appearing in the proof is nothing but the category of $k(x)$-vector spaces with an integrable connection $(M,\nabla:M\to M\otimes_{k(x)}\Omega^1_{k(x)/k})$.
They played a role for example in~\cite{MR667751,MR2211153}.
In particular, there is a canonical fiber functor to $k(x)$-vector spaces (which forgets the connection).
\end{rmk}

\begin{rmk}
\label{rmk:regular-singularities}%
Restricting to the full subcategory $\Dbrh(\cD_X)$ of complexes whose total cohomology is a regular holonomic $\cD$-module, one obtains another generically simple constructible system.
Moreover, if $k=\CC$, it is equivalent to constructible analytic sheaves, by the Riemann-Hilbert correspondence.
\end{rmk}

\subsection{Mixed Hodge modules}
Here we work with complex varieties.
For a regular variety $Y$ we denote by $\Dbmhm{Y}$ the bounded derived category of mixed (algebraic) Hodge modules on~$Y$ in the sense of Saito~\cite{MR1047415}.
According to~\cite[Theorem~0.1]{MR1047415} there is a conservative tt-functor $\rat:\Dbmhm{Y}\to\Dbc(Y;\QQ)$ which is compatible with the usual functoriality in~$Y$.
\begin{prop}
\label{sta:dbmhm}%
Assume $X$ may be embedded in a regular variety.
The assignment $V\mapsto\Dbmhm{V}$ for regular $V\subseteq X$ extends to a generically simple constructible system on~$X$.
\end{prop}
\begin{proof}
The argument for extending the assignment to a pseudo-functor $\Sub\op\to\ttCat$ is the same as in \Cref{sta:dbh-gscs}, using the functor $\rat$ instead of $\nu$.
The fact that this is a constructible system then follows from \Cref{sta:cs-sub} and \Cref{sta:Dbcadic-gscs}, using again the functor~$\rat$.
For \bref{ax:simple} let $V$ be a connected regular subvariety of~$X$ with generic point~$x$.
Since filtered colimits commute with $\Db$ (for example by~\cite[Lemma~2.6]{gallauer:tannaka-dmet}), the canonical tt-functor $\Db(\MHM(x))\to \Db(\MHM)(x)$ is an equivalence.
But $\MHM(x)$ is the filtered colimit, over $x\in U\subseteq V$ open, of the $\QQ$-linear Tannakian categories $\VMHSadm(U)$ of admissible variations of mixed Hodge structures on~$U$~\cite[p.\,313]{MR1047415}.
We conclude again with \Cref{sta:cs-gen-simple-criterion}.
\end{proof}

\begin{rmk}
\label{rmk:Dbgm-MHM}%
Another example of a generically simple constructible system should be formed by mixed Hodge modules of geometric origin~\cite[\S\,7]{saito:formalism}.
\end{rmk}

\subsection{Motivic sheaves}
\label{sec:motivic-sheaves}
Let $\EE$ be a characteristic zero field and denote, for any scheme $X$, by $\DM(X;\EE)$ a `good' category of motivic sheaves on $X$ with coefficients in~$\EE$.
To fix our ideas we take $\DM(X;\EE)$ to mean \'etale motivic sheaves $\DAet(X;\EE)$ in the sense of~\cite{ayoub:etale-realization}.
By~\cite{cisinski-deglise:DM}, we could equivalently consider Beilinson motives.
Let $\DMc(X;\EE)$ denote the subcategory of constructible motivic sheaves, that is, the thick subcategory generated by the motives of separated smooth $X$-schemes and their negative Tate twists.
Equivalently (at least if $X$ is finite dimensional), these are precisely the compact objects~\cite[Proposition~8.3]{ayoub:etale-realization}.

Now fix a base field~$k$.
Then $\DMc(k;\EE)$ is nothing but Voevodsky's triangulated category of geometric motives.
Recall that one expects a `motivic t-structure' on $\DMc(k;\EE)$, with heart $\MM(k;\EE)$, the abelian category of mixed motives over~$k$.
This should be a Tannakian category with fiber functors induced by the classical cohomology theories.
More optimistically one might hope that $\DMc(k;\EE)$ is equivalent to $\Db(\MM(k;\EE))$, which would imply that $\DMc(k;\EE)$ is a simple tt-category.
We refer the reader to~\cite[\S\,3]{gallauer:tannaka-dmet} for an extended discussion of this point.
(The main ingredient in proving the simplicity of~$\Db(\MM(k;\EE))$ is, as in \Cref{sta:cs-gen-simple-criterion}, \cite[Theorem~2.4]{gallauer:tannaka-dmet}.)

In the sequel we will assume this holds for all fields relevant to the discussion.
Namely, let $X$ be a quasi-projective $\kk$-variety.\footnote{We restrict to quasi-projective varieties only to avail ourselves of convenient references.}
\begin{hyp}
\label{hyp:DM-gs}
The tt-category $\DMc(\kk(x);\EE)$ is simple for every $x\in X$.
\end{hyp}\noindent

In addition we assume that $k$ and $\EE$ satisfy one of the following two conditions:
\begin{enumerate}[(1)]
\item $k\subseteq\CC$ and $\EE$ is arbitrary, or
\item $\ell$ is a prime number invertible in $\kk$ and $\EE$ is a finite extension of $\QQ_\ell$.
\end{enumerate}
In the first case let $\Re_X:\DMc(X;\EE)\to\Dbc(X;\EE)$ denote the Betti realization~\cite[D\'efinition~2.1]{ayoub:betti}, and in the second case the \'etale realization~\cite[D\'efinition~9.6]{ayoub:etale-realization}.
(When both conditions are satisfied, the choice of realization doesn't matter for what follows.)

\begin{prop}
\label{sta:dm-gscs}%
Under \Cref{hyp:DM-gs} the pseudo-functor $\DMc(-;\EE)$ is a generically simple constructible system on $X$.
\end{prop}
\begin{proof}
Fix a locally closed subvariety $Y\subseteq X$.
For $x\in Y$ the canonical functor
\begin{equation}
\label{eq:pt-vs-pt}
\DMc(x;\EE)\isoto\DMc(k(x);\EE)
\end{equation}
is an equivalence, by~\cite[Corollaire~3.22]{ayoub:etale-realization}.
Therefore, \Cref{hyp:DM-gs} implies directly \bref{ax:simple}.
It follows from the equivalence in~(\ref{eq:pt-vs-pt}) together with~\cite[Proposition~3.24]{ayoub:etale-realization} that the family $\rho_x:\DMc(Y;\EE)\to\DMc(x;\EE)$ is conservative when $x$ runs through all points of~$Y$.
Since non-trivial tt-functors out of simple tt-categories are necessarily conservative we deduce that the family $\Re_x\circ\rho_x:\DMc(Y;\EE)\to\Dbc(x;\EE)$ is also conservative, where we denote by $\Re_x$ the functor
\[
\DMc(x;\EE)=\textup{2-colim}_V\,\DMc(V;\EE)\xto{\textup{2-colim}_V\Re_V}\textup{2-colim}_V\,\Dbc(V;\EE)=\Dbc(x;\EE)
\]
But $\Re_x\circ\rho_x=\rho_x\circ\Re_Y$ so that $\Re_Y$ is necessarily conservative too.

Next, the functor $\DMc(-;\EE)$ underlies a fully-fledged six-functor formalism~\cite[\S\,8]{ayoub:etale-realization}, and by~\cite[Th\'eor\`eme~9.6]{ayoub:etale-realization} and~\cite[Th\'eor\`eme~3.19]{ayoub:betti}, the realization functor~$\Re$ commutes with the six functors.
By the preceding argument, we may now apply \Cref{sta:lift-cs-conservative} (and \Cref{sta:Dbc-gscs,sta:Dbcadic-gscs}) to conclude.
\end{proof}

\section{Support theory}
\label{sec:support}

We first define, for every generically simple constructible system, a support theory (\Cref{sec:supp-constr-object}).
We then show that it is the universal support theory in the sense of Balmer~\cite{balmer:spectrum}, and we derive consequences (\Cref{sec:main-result}).
In particular, we prove \Cref{thrm:intro-spc} of the introduction.

\subsection{The support of a constructible object}
\label{sec:supp-constr-object}
Let $\CS:\V\op\to\ttCat$ be a generically simple constructible system as in \Cref{cs}. 
Also fix an object $X\in\V$.

\begin{defn}
\label{defn:support}%
Let $M\in\CS(X)$.
We define its \emph{support} as a subset of the points of the scheme~$X$:
\[
\Supp(M):=\{x\in X\mid M\notin\pr_x\}
\]
\end{defn}

\begin{rmk}
\label{rmk:supp-alternative}%
In other words, $x\notin\Supp(M)$ if and only if $M$ vanishes on a locally closed subset containing~$x$.
This follows immediately from the definitions.
\end{rmk}

\begin{rmk}
\label{rmk:constructible-topology}%
In the following we will use the constructible topology.
For a topological space~$T$ we denote by $2^T$ the set of its subsets.
Recall that (at least for $T$ noetherian) $V\in 2^T$ is called constructible if it belongs to the Boolean subalgebra generated by open subsets.
Equivalently, if $V$ is a finite union of locally closed subsets.
The constructible topology $\cons{X}$ on the scheme~$X$ is generated by the constructible subsets~\cite[1.9.13]{MR173675}.
Hence the open subsets of $\cons{X}$ are the unions of locally closed subsets.
The closed subsets are intersections of constructible subsets.

The space $\cons{X}$ is a Stone space: compact, Hausdorff and totally disconnected.
In fact, it is the image of $X$ under the right adjoint to the inclusion of Stone spaces into coherent spaces.
An open subset of $\cons{X}$ is (quasi-)compact iff clopen in $\cons{X}$ iff constructible in $X$.
It follows that $(\cons{X})^*=\cons{X}$ where $(-)^*$ denotes the Hochster dual (cf.\ \Cref{exa:spc-lattice}).
In other words, the `Thomason subsets'~\cite{balmer:spectrum} of $\cons{X}$ coincide with the open ones.
\end{rmk}

\begin{prop}
\label{sta:support-constructible}%
Let $M\in\CS(X)$.
Then $\Supp(M)\subseteq X$ is a constructible subset.
\end{prop}
\begin{proof}
Let $x\in\Supp(M)$.
We are going to show that there exists a locally closed subset $x\in V\subseteq\Supp(M)$. 
Replacing $X$ by $\cl(x)$ and $M$ by $M|_{\cl(x)}$, we may assume $X$ integral with generic point~$x$.
Replacing $X$ by an open dense subset we may assume that $X$ is regular, connected, and that $M$ is lisse.
By the last condition in \bref{ax:cons}, we then have $\Supp(M)=X$.

This proves that $\Supp(M)\subseteq\cons{X}$ is open.
\Cref{rmk:supp-alternative} implies that $\Supp(M)\subseteq\cons{X}$ is closed.
We conclude with \Cref{rmk:constructible-topology}.
\end{proof}

\begin{lem}
\label{sta:support}%
The map $\Supp:\mathrm{Obj}(\CS(X))\to 2^{\cons{X}}$ is a support datum in the sense of~\cite{balmer:spectrum}.
\end{lem}
\begin{proof}
By \Cref{sta:support-constructible}, $\Supp(M)\subseteq\cons{X}$ is closed for every $M\in\CS(X)$.
All other required properties follow easily from the $\pr_x$ being prime tt-ideals.
\end{proof}

\begin{notn}
Of course, we may extend supports to arbitrary sets of objects $\cK\subseteq\CS(X)$:
\[
\Supp(\cK):=\bigcup_{M\in\cK}\Supp(M)\subseteq\cons{X}
\]
Conversely, if $V\subseteq \cons{X}$ is a subset, we define $\cI_V\subseteq\CS(X)$ to be the full subcategory spanned by
\[
\{M\in\CS(X)\mid \Supp(M)\subseteq V\}.
\]
\end{notn}
The following are the `easy' properties of these associations which will subsequently be shown to be inverses to each other. 
\begin{lem}
\label{sta:ideal-support-easy}%
\begin{enumerate}[(a)]
\item For any set $\cK\subseteq\CS(X)$, the subset $\Supp(\cK)\subseteq \cons{X}$ is open.
\item For any $V\subseteq\cons{X}$, the subcategory $\cI_V\subseteq \CS(X)$ is a radical tt-ideal.\footnote{A tt-ideal $\cI$ is said to be radical if $M^{\otimes n}\in\cI$ implies $M\in\cI$.
  If $\cI$ is a tt-ideal, its radical $\sqrt{\cI}:=\{M\mid M^{\otimes n}\in\cI\}$ is a radical tt-ideal~\cite[Lemma~4.2]{balmer:spectrum}.
  More generally, for any set $\cK$ we denote by $\radtt{\cK}$ the radical tt-ideal generated by~$\cK$.}
\item For any set $\cK\subseteq\CS(X)$ the following inclusion holds:
\[
\radtt{\cK}\subseteq\cI_{\Supp(\cK)}
\]
\item For any subset $V\subseteq\cons{X}$, we have
\[
 \Supp(\cI_V)\subseteq V,
\]
with equality if $V$ is open.
\end{enumerate}
\end{lem}
\begin{proof}
The first statement follows immediately from \Cref{sta:support-constructible}.
The second statement is a consequence of the fact that we have $\cI_V=\bigcap_{x\notin V}\pr_x$ which is an intersection of prime tt-ideals hence a radical tt-ideal itself.
The inclusion in the third statement follows immediately.
The inclusion in the last statement is obvious.
But if $V\subseteq \cons{X}$ is open and $x\in V$ there exists a locally closed subset $W\subseteq \zar{X}$ with $x\in W\subseteq V$.
With the notation of \Cref{notn:extension-by-zero}, we have $\one_W\in\left(\cap_{y\notin V}\pr_y\right)\backslash\pr_x$ which gives the reverse inclusion.
\end{proof}

\subsection{Classification}
\label{sec:main-result}

Here are our two main results.

\begin{thrm}
\label{sta:classification}%
Let $\CS$ be a generically simple constructible system on $\V$, and let $X\in\V$.
The assignments 
\[
\begin{tikzcd}[column sep=50pt]
\TI(\CS(X))
\ar[r, shift left=1.5, "\Supp", "\sim" below]
&
\Op(\cons{X})
\ar[l, shift left=1.5, "\cI"]
\end{tikzcd}
\]
between $\TI(\CS(X))$, the set of tt-ideals in $\CS(X)$, and $\Op(\cons{X})$, the set of open subsets of $\cons{X}$, are inclusion-preserving inverse bijections.
\end{thrm}

\begin{cor}
\label{sta:main}%
Let $\CS$ be a generically simple constructible system on $\V$.
The composite functor
\[
\V\xto{\CS}\ttCat\op\xto{\Spc}\TopSpc
\]
is naturally isomorphic to $X\mapsto \cons{X}$.
In particular, for every $X\in\V$ we have
\[
\Spc(\CS(X))\cong \cons{X}.
\]
\end{cor}
The spectrum $\Spc(\cT)$ of a tt-category~$\cT$ is the set of prime tt-ideals in~$\cT$ with a basis of closed subsets given by~$\supp(t)=\{\cP\mid t\notin\cP\}$.
We refer the reader to~\cite{balmer:spectrum} where it was introduced; see also \Cref{exa:spc-lattice}.

\begin{rmk}
\label{rmk:radical-automatic}%
Let $X\in\V$.
It follows from \Cref{sta:classification} that every tt-ideal in $\CS(X)$ is radical.
Indeed, \Cref{sta:ideal-support-easy} says that $\cI_V$ is a radical tt-ideal for every subset $V\subseteq X$.
\end{rmk}

\begin{proof}[Proof of \Cref{sta:main}]
By \Cref{sta:support} and the universality of $\Spc$ proved in~\cite[Theorem~3.2]{balmer:spectrum}, we get, for each $X\in\V$, a continuous map $\phi_X:\cons{X}\to\Spc(\CS(X))$.
Explicitly, it is given by $\phi_X(x)=\{M\in\CS(X)\mid x\notin\Supp(M)\}=\pr_x$.
It follows from \Cref{sta:primes-functoriality} that the $\phi_X$ are the components of a natural transformation $\phi:\cons{(-)}\to\Spc\circ\CS$.
By \Cref{rmk:radical-automatic}, all tt-ideals in $\CS(X)$ are radical.
\Cref{sta:classification} then says precisely that the criterion of~\cite[Corollary~5.2]{MR2280286} is satisfied so that $\phi_X$ is a homeomorphism.
\end{proof}

The remainder of this section is devoted to the proof of \Cref{sta:classification}.
We fix a generically simple constructible system~$\CS:\V\op\to\ttCat$ and $X\in\V$.
As in \Cref{sta:classification}, we denote, for a topological space $T$, the set of open subsets by $\Op(T)$.
And for a tt-category $\cT$, the set of (resp.\ radical) tt-ideals will be denoted by $\TI(\cT)$ (resp.\ $\RI(\cT)$).
\begin{lem}
\label{sta:radical-ideals-recollement}%
Let $i:Z\into X$ be a closed immersion in $\V$, with open complement $j:U\into X$.
Then the maps
\begin{align*}
  \TI(\CS(X))&\longrightarrow \TI(\CS(U))\times\TI(\CS(Z))\\
  \RI(\CS(X))&\longrightarrow \RI(\CS(U))\times\RI(\CS(Z))\\
       \cK&\longmapsto (j^*\cK,i^*\cK)
\end{align*}
are bijective. 
Moreover, the following square commutes (also with $\TI$ instead of $\RI$):
\[
\begin{tikzcd}[column sep=50pt]
\RI(\CS(X))
\ar[d, "\sim", "j^*\times i^*" left]
\ar[r, "\Supp"]
&
\Op(\cons{X})
\ar[d, "\sim", "j^{-1}\times i^{-1}" left]
\\
\RI(\CS(U))\times\RI(\CS(Z))
\ar[r, "\Supp\times\Supp"]
&
\Op(\cons{U})\times\Op(\cons{Z})
\end{tikzcd}
\]
\end{lem}
\begin{proof}
Note that $j^*$ and $i^*$ being retractions (to $j_!$ and $i_*$, respectively), they are both full and essentially surjective.
It follows that $j^*\times i^*$ indeed sends a tt-ideal to a tt-ideal.

Next, for $M\in\CS(X)$ and $\cI\in\TI(\CS(X))$ we have
\begin{equation}
\label{eq:open-closed}
j^*M\in j^*\cI \Longleftrightarrow j_!j^*M\in\cI \quad\text{(resp. }i^*M\in i^*\cI \Longleftrightarrow i_*i^*M\in\cI\text{)}.
\end{equation}
Indeed, one direction being obvious, assume $j^*M\in j^*\cI$.
Hence there is $N\in\cI$ with $j^*M=j^*N$.
By the projection formula of \Cref{sta:projection-formulae} we have
\[
j_!j^*M=j_!j^*N=j_!j^*\one\otimes N\in \cI
\]
as required.
The same argument applies to the parenthesized statement in~\eqref{eq:open-closed}.
Note that the projection formul\ae{} also show that $j^*\times i^*$ restricts to radical tt-ideals.

Together with the triangle $j_!j^*M\to M\to i_*i^*M\to^+$ of the recollement (\Cref{recollement-bbd}) we deduce from~\eqref{eq:open-closed} that
\[
M\in \cI\Longleftrightarrow j^*M\in j^*\cI\text{ and }i^*M\in i^*\cI
\]
and this easily gives the bijections in the statement.
(The inverse to $j^*\times i^*$ is given by $(\cL,\cM)\mapsto (j^*)^{-1}(\cL)\cap(i^*)^{-1}(\cM)$.)

For the commutativity of the diagram, assume $\cI\in\TI(\CS(X))$ and let us show that $\Supp(\cI)\cap U=\Supp(j^*\cI)$.
(Again, the argument on $Z$ is the same.)
This follows immediately from the commutative triangle, for each $x\in U$:
\[
\begin{tikzcd}
\CS(X)
\ar[r, "\rho_x"]
\ar[d, "j^*" left]
&
\CS(x)
\\
\CS(U)
\ar[ru, "\rho_x" right]
\end{tikzcd}
\]
\end{proof}

\begin{proof}[Proof of \Cref{sta:classification}]
In \Cref{sta:ideal-support-easy} we already showed $\Supp\circ\cI=\id_{\Op(\cons{X})}$ and it suffices to show $\cI\circ\Supp=\id_{\TI(\CS(X))}$. 
Thus let $\cK\in\TI(\CS(X))$ be a tt-ideal.
By \Cref{sta:ideal-support-easy} again, we know that $\cK\subseteq\cI\circ\Supp(\cK)$, and it suffices to show the reverse inclusion.
So let $M\in\cI_{\Supp(\cK)}=\cap_{\cK\subseteq\pr_x}\pr_x$, and let us show that $M\in\cK$.
By \Cref{sta:radical-ideals-recollement} and noetherian induction, it suffices to find a non-empty open subset $U\subseteq X$ such that $M|_U\in\cK|_U\subseteq\CS(U)$.
Consequently we may assume $X$ is regular, connected, and $M\in\CSls(X)$, by \bref{ax:cons}.
If $M=0$ then certainly $M\in\cK$ so let us assume $M\neq 0$.
If $x\in X$ denotes the generic point then \bref{ax:cons} implies that $M\notin\pr_x$.
In particular, there exists $L\in\cK$ with $L\notin\pr_x$.
By simplicity of $\CS(x)$ we must have $\rho_x(M)\in\ttid{\rho_x(L)}$\footnote{Here, $\ttid{S}$ denotes the tt-ideal generated by a set $S$.} and hence, by \Cref{sta:spreading-out} below, there exists a non-empty open $U\subseteq X$ such that $M|_U\in\ttid{L|_U}\subseteq\cK|_U\subseteq\CS(U)$.
This completes the proof.
\end{proof}

The following result is formulated in the language introduced in~\cite[\S\,8]{gallauer:tt-fmod}.
A special case, enough for our application in the proof of \Cref{sta:classification}, is when $\cT_\bullet$ is a filtered diagram with 2-colimit $\cT$.
\begin{lem}
\label{sta:spreading-out}%
Let $\cT_\bullet:I\to\ttCat$ be a pseudo-functor with $I$ conjoining, and $f:\cT_\bullet\to\cT$ a pseudo-natural transformation.
Assume that $f$ is surjective on morphisms and detects isomorphisms~\cite[Definition~8.4]{gallauer:tt-fmod}.
If $a,b\in\cT$ satisfy $a\in\ttid{b}$ there exist $i$ and lifts $a_i\in\cT_i$ of $a$ and $b_i\in\cT_i$ of $b$, such that $a_i\in\ttid{b_i}$.
\end{lem}
\begin{proof}
The relation $a\in\ttid{b}$ means that $a$ can be built from $b$ in finitely many steps.
Each of these steps can be successively lifted to some $\cT_i$.
(For a similar argument the reader is referred to the proof of~\cite[Proposition~8.5]{gallauer:tt-fmod}.)
\end{proof}

\section{Monoidal topological reconstruction}
\label{sec:toprec}

In this section we discuss applications of our results in \Cref{sec:main-result} to reconstructing varieties~$X$ from $\CS(X)$.
In \Cref{sta:main} we computed $\Spc(\CS(X))=\cons{X}$, the underlying set of~$X$ with the constructible topology.
The identity map $\cons{X}\to\zar{X}$ exhibits the former as the universal Stone space mapping to the latter.
Conversely, we may recover $\zar{X}$ from $\cons{X}$ together with the specialization relations in~$\zar{X}$~\cite[II.4]{MR698074}.
Our first goal (\Cref{sec:smashing-spectrum}) is to formalize this last process for tt-categories and their spectra.
In \Cref{sec:smashing-ideals} we specialize to generically simple constructible systems and prove \Cref{thrm:intro-spcs} of the introduction.

\subsection{The smashing spectrum}
\label{sec:smashing-spectrum}

\begin{rmk}
\label{rmk:spcs-informal}%
We start with an informal account of what the ensuing (rather abstract) discussion is supposed to achieve.
Let $U\subseteq \zar{X}$ be an open subset.
Then there is an object~$\one_U\in\CS(X)$ that, as we will prove in \Cref{sec:smashing-ideals}, is characterized by the following three properties:
\begin{enumerate*}[(1)]
\item Its support is precisely~$U$,
\item it comes with a morphism $\mu_U:\one_U\to\one$, and
\item the morphism $\mu_U\otimes \id:\one_U\otimes\one_U\isoto\one_U$ is invertible.
\end{enumerate*}
We call $\mu_U:\one_U\to\one$ an \emph{open idempotent} if it exhibits $\one_U$ as a $\otimes$-idempotent as in the last condition.
An open idempotent $\mu_{U'}$ that factors through $\mu_U$ is `smaller'.
Indeed, this is equivalent to $U'\subseteq U$.
And if we restrict to complements of irreducible closed subsets we obtain the specialization relation on the points of~$\zar{X}$.
In fact, we may recover $\zar{X}$ directly by letting the open subsets be the supports of open idempotents.

This makes sense in general tt-categories~$\cT$.
Open idempotents form a distributive lattice, and we call the associated spectrum the \emph{smashing spectrum} of~$\cT$ and denote it by $\Spcs(\cT)$.
(The name is justified by the correspondence between idempotents and smashing tt-ideals which we recall.)
Under certain restrictive hypotheses on~$\cT$, the canonical map $\Spc(\cT)\to\Spcs(\cT)$ exhibits the former as the universal Stone space mapping to the latter.
In particular, in the case of $\cT=\CS(X)$, we recover the canonical map $\cons{X}\to\zar{X}$.

Many other works have explored similar ideas, notably~\cite{MR3147415,MR2806103,MR4064108,MR4093065}.
\end{rmk}

\begin{notn}
\label{notn:T-small}%
Throughout this section $\cT$ denotes an essentially small tt-category.
\end{notn}

\begin{rmk}
\label{rmk:smashing-warning}%
We feel the need to accompany \Cref{notn:T-small} with a `warning'.
Originally, smashing ideals were studied mostly in the context of rigidly compactly generated tt-categories, particularly the stable homotopy category in algebraic topology.
In that context there are few smashing ideals in the subcategory of compact (equivalently, rigid) objects.
In other words, the idempotents are typically not compact.
Nevertheless, there \emph{are} small tt-categories in which the set of smashing ideals is informative, and it is these that we are ultimately interested in.
For a more precise `warning' see \Cref{rmk:smashing-topology}.
\end{rmk}

A convenient modern reference for the next couple remarks is~\cite{MR2806103} (except, again, that we don't assume our triangulated categories to be closed under coproducts).

\begin{notn}
\label{notn:smashing-ideal}
Recall that a tt-ideal $\cK$ is called \emph{smashing} if it is the kernel of a Bousfield localization $L:\cT\to\cT$ and if the right orthogonal complement~$\cK^\bot$ is a tt-ideal as well.
We obtain an exact triangle~$\Delta_{\cK}$ of endofunctors $\Gamma\to\id\to L$ with $\Gamma$ a colocalization, and the relations
\[
\ker(\Gamma)=\im(L)=\cK^\bot,\qquad \ker(L)=\im(\Gamma)=\cK.
\]
\end{notn}

\begin{rmk}
\label{rmk:smashing-idempotents}%
Of particular interest is the evaluation of $\Delta_{\cK}$ at the unit~$\one$.
The triangle
\[
\Delta_{\cK}(\one)\ =\ 
\left(
\oid:=\Gamma\one\to\one\to L\one=:\cid
\right)
\]
is idempotent in the sense that $\oid\otimes\oid\isoto\oid$ and $\cid\isoto\cid\otimes\cid$.
We call $\oid\to\one$ (or just $\oid$) the \emph{open idempotent} and $\one\to\cid$ (or just $\cid$) the \emph{closed idempotent}.
Conversely, starting with, say, an open idempotent~$\oid\to\one$, the completed exact triangle
\[
\oid\to\one\to\cid
\]
is idempotent and defines a smashing ideal $\ker(\cid\otimes-)=\ttid{\oid}=\oid\otimes\cT$.
(We sometimes denote the closed idempotent complementary to an open idempotent by $\oid^{\bot}:=\cid$, and similarly ${}^{\bot}\cid:=\oid$.)
This sets up a 1-to-1 correspondence between smashing ideals $\SI(\cT)$, open idempotents $\OID(\cT)$ (up to isomorphism) and closed idempotents $\CID(\cT)$ (up to isomorphism).
Here, a morphism $\oid\to\oid'$ of open idempotents is a morphism which makes the obvious triangle commute, and similarly for closed idempotents.
\end{rmk}

\begin{lem}
\label{sta:idempotent-morphisms}%
Let $\oid,\oid'$ be two open idempotents with complementary closed idempotents $\cid,\cid'$, respectively.
The following are equivalent:
\begin{enumerate}[(i)]
\item There exists a morphism of open idempotents $\oid\to\oid'$.
\item $\oid\in\ttid{\oid'}$
\item $\oid\otimes\cid'=0$
\item $\oid\otimes\oid'\isoto\oid$
\end{enumerate}
Moreover, if these conditions are satisfied then there is a \emph{unique} morphism of open idempotents $\oid\to\oid'$.
\end{lem}
\begin{proof}
The last three conditions are clearly equivalent, and they yield a morphism of open idempotents $\oid\isofrom\oid\otimes\oid'\to\oid'$.
Conversely, assume $\oid\to\oid'$ is such a morphism of open idempotents and let us prove that $\oid\otimes\cid'=0$.
Consider the exact sequence
\[
\hom(\oid\otimes\cid',\cid[-1]\otimes\cid')\to\hom(\oid\otimes\cid',\oid\otimes\cid')\to\hom(\oid\otimes\cid',\cid').
\]
The first term vanishes and the canonical morphism $\oid\otimes\cid'\to\cid'$ is zero since it factors through $\oid'\otimes\cid'=0$ thus the claim.

Uniqueness of a morphism $\oid\to\oid'$ of open idempotents follows from the fact that the group $\hom(\oid,\cid'[-1])\cong\hom(\oid\otimes\oid',\cid'[-1])$ vanishes.
\end{proof}
\begin{rmk}
\label{rmk:posets}%
Define meet and join operations on smashing ideals in terms of their associated idempotents: 
\[
\cK\wedge\cK' = \ttid{\oid(\cK)\otimes\oid(\cK')},\qquad \cK\vee\cK' = \ker(\cid(\cK)\otimes\cid(\cK')\otimes-).
\]
In fact, with obvious orderings of open and closed idempotents the correspondence of \Cref{rmk:smashing-idempotents} yields isomorphisms of lattices $\SI(\cT)\cong\OID(\cT)\cong\CID(\cT)$~\cite[Proposition~3.11]{MR2806103}.
\end{rmk}

\begin{notn}
\label{notn:lattice}%
Recall that a distributive lattice is a non-empty poset (viewed as a category) which has finite limits and colimits and these distribute.
A morphism of distributive lattices is an exact functor.
They thus form a category denoted $\DLat$.

The initial distributive lattice is the ordered set $\bool=\{0,1\}$.
The spectrum of a distributive lattice~$A$ is the set of morphisms
\[
\Spec(A):=\hom_{\DLat}(A,\bool)
\]
with the topology generated by sets of the form
\[
\supp(a):=\{p\mid p(a)=1\},\quad a\in A.
\]
This sets up an anti-equivalence between distributive lattices and the category $\CohSp$ of coherent spaces with coherent maps~\cite[Corollary~II.3.4]{MR698074}.
The inverse functor is given by sending a coherent space~$X$ to the lattice of its quasi-compact open subsets.
\end{notn}

\begin{exa}
\label{exa:spc-lattice}
The set of radical tt-ideals $\RI(\cT)$ ordered by inclusion is a distributive lattice.
In fact, it is a compactly generated poset whose compact objects form a distributive lattice.
In particular, it satisfies an infinite form of distributivity (it is a coherent frame).
To see this, note that the inclusion $\RI(\cT)\into\TI(\cT)$ admits a left adjoint $\sqrt{-}$ and that $\TI(\cT)$ is clearly a compactly generated poset.
(It is the Ind-completion of the lattice of finitely generated tt-ideals.)
Moreover, as the directed union of radical tt-ideals is radical, it follows that $\RI(\cT)$ is compactly generated, with compact objects $\RIc(\cT)$ the radicals of finitely generated tt-ideals.

A morphism of distributive lattices $p:\RIc(\cT)\to\bool$ extends uniquely to a colimit preserving exact functor $\bar{p}:\RI(\cT)\to\bool$ which may be identified with its kernel $(\bar{p})^{-1}(0)\subseteq\RI(\cT)$.
The latter is a prime ideal and hence principal (as in every frame) thus a further identification with a radical ideal in~$\cT$ and this is easily seen to be a prime ideal in~$\cT$.
Therefore, the spaces $\Spec(\RIc(\cT))$ and $\Spc(\cT)$ have the same underlying set.
Their topologies are dual in the sense of Hochster~\cite{hochster:spectral}, that is, they correspond to each other under the equivalence $X\mapsto X^*$ which fits into the following commutative square:
\[
\begin{tikzcd}
\DLat\op
\ar[r, <->, "\sim"]
\ar[d, <->, "\sim" right, "(-)\op" left]
&
\CohSp
\ar[d, <->, "\sim" left, "(-)^*" right]
\\
\DLat\op
\ar[r, <->, "\sim"]
&
\CohSp
\end{tikzcd}
\]
\end{exa}

\begin{lem}
\label{sta:smashing-lattice}%
The lattice of smashing ideals $\SI(\cT)$ is distributive, and $\cK\mapsto\sqrt{\cK}$ defines an injective homomorphism of distributive lattices $\sqrt{-}:\SI(\cT)\to\RIc(\cT)$.
\end{lem}
\begin{proof}
We factor the functor in the statement as $\SI(\cT)\into\TIc(\cT)\xto{\sqrt{-}}\RIc(\cT)$ and note that, by \Cref{exa:spc-lattice}, the second functor is exact.
For the first functor we need to show that $\cK\wedge\cK'=\cK\cap\cK'$ and $\cK\vee\cK'=\cK+\cK'$ for any smashing ideals~$\cK,\cK'$.
In each case, only one inclusion requires proof.
Thus let $t\in\cK\cap\cK'$.
We then have $\oid\otimes t\isoto t$ and $\oid'\otimes t\isoto t$ (with hopefully obvious notation), and therefore $\oid\otimes\oid'\otimes t\isoto t$ hence $t\in\cK\wedge\cK'$.
Now let $t\in\cK\vee\cK'$, that is, $\cid\otimes\cid'\otimes t=0$.
Let $\oid\vee\oid'$ be the open idempotent complementary to $\cid\otimes\cid'$.
By~\cite[Theorem~3.13]{MR2806103}, we have an exact Mayer-Vietoris triangle
\begin{equation}
\label{eq:mayer-vietoris}
\oid\otimes\oid'\otimes t\to(\oid\otimes t)\oplus(\oid'\otimes t)\to(\oid\vee\oid')\otimes t,
\end{equation}
and the last term is isomorphic to~$t$.
But the first term is in $\cK\cap\cK'$ and the second is in $\cK+\cK'$ thus the claim.

For distributivity in~$\SI(\cT)$, and with analogous notation, we need to show that
\[
(\oid\vee\oid')\otimes\oid''\in\cK\cap\cK''+\cK'\cap\cK''.
\]
Replacing $t$ by $\oid''$ in the triangle~(\ref{eq:mayer-vietoris}), the claim follows.

Finally, we show that $\sqrt{-}:\SI(\cT)\to\RIc(\cT)$ is injective.
This amounts to the following claim.
If $\alpha:\oid\to\oid'$ is a morphism of open idempotents, with $\oid'\in\radtt{\oid}$ then $\alpha$ is invertible.
But as $\oid'$ is an idempotent, $\oid'\in\radtt{\oid}$ implies $\oid'\in\ttid{\oid}$ and the claim follows from \Cref{sta:idempotent-morphisms}.
\end{proof}

\begin{defn}
\label{defn:smashing-spectrum}%
The \emph{smashing spectrum of $\cT$} is the spectrum of the lattice of smashing ideals:
\[
\Spcs(\cT):=\Spec(\SI(\cT))
\]
In other words, the underlying set consists of the homomorphisms $p:\SI(\cT)\to\bool$ and a basis for the topology is given by the sets $\{p\mid p(\cK)=1\}$ for $\cK\in\SI(\cT)$.
\end{defn}

We spell out some direct consequences of the definition.
\begin{prop}
\label{sta:spcs-space}%
\begin{enumerate}[(a)]
\item The space $\Spcs(\cT)$ is coherent.
\item The association $\cT\mapsto\Spcs(\cT)$ extends to a functor $\Spcs:\ttCat\op\to\CohSp$.
\item There is a canonical natural transformation of functors $\epsilon:\Spc\to\Spcs$.
\end{enumerate}
\end{prop}
\begin{proof}
The first statement is clear.
On the other hand, the association $\cT\mapsto\SI(\cT)$ defines a functor $\ttCat\to\DLat$, acting on tt-functors $F:\cT\to\cT'$ by sending a closed (say) idempotent $\cid\in\cT$ to $F(\cid)\in\cT'$, thus the second statement.
Moreover, the association $\cT\mapsto\RIc(\cT)$ similarly defines a functor $\ttCat\to\DLat$, acting on tt-functors $F:\cT\to\cT'$ by sending $\cK$ to $\radtt{F(\cK)}$.
It is then clear that the homomorphism $\sqrt{-}:\SI\to\RIc$ of \Cref{sta:smashing-lattice} defines a natural transformation.
Composing with the spectrum equivalence we obtain a natural transformation
\[
\epsilon:\Spec(\RIc(-))\to\Spec(\SI(-))=\Spcs(-)
\]
and it suffices to show that $\epsilon_{\cT}$ is coherent also for the dual topology~$\Spc(\cT)$ (\Cref{exa:spc-lattice}).
But given a smashing ideal $\cK\in\SI(\cT)$ corresponding to the idempotent triangle $\oid\to\one\to\cid$ we have $\epsilon_{\cT}^{-1}(\supp(\cK))=\supp(\oid)=\Spc(\cT)\backslash\supp(\cid)$ which is a quasi-compact open in the dual topology.
\end{proof}

\begin{rmk}
\label{rmk:spcs-classifying-support}%
Consider a pair $(X,\sigma)$ where $X$ is a topological space and $\sigma$ a function that assigns to a smashing ideal $\cK\in\SI(\cT)$ a subset $\sigma(\cK)\subseteq X$.
It is called a \emph{support datum} on $\SI(\cT)$ if $\sigma$ is a lattice homomorphism to the frame~$\Op(X)$ of open subsets.
By the $\Spec$-$\Op$ adjunction~\cite[Theorem~II.1.4]{MR698074}, $\sigma$ is the same as a continuous map $X\to\Spcs(\cT)$.
In other words, $(\Spcs(\cT),\supp)$ is the universal support datum on~$\SI(\cT)$.
This is analogous to~\cite[Theorem~3.2]{balmer:spectrum} or~\cite[Theorem~4.3]{MR2280286}.
It follows that it is also \emph{classifying}, that is, the analogues of~\cite[Theorem~4.10]{balmer:spectrum} or~\cite[Corollary~5.2]{MR2280286} hold.
\end{rmk}

\begin{rmk}
\label{rmk:smashing-topology}%
In general the smashing spectrum can be a rather poor invariant.
For example as seen in the proof of \Cref{sta:spcs-space}, if $\Spc(\cT)$ is connected then $\Spcs(\cT)$ is a singleton space.
The following definition singles out a class of tt-categories~$\cT$ for which $\Spcs(\cT)$ is at least as rich an invariant as~$\Spc(\cT)$.
\end{rmk}

\begin{defn}
\label{defn:enough-idempotents}%
If $\oid$ is an open idempotent and $\cid$ is a closed idempotent then we call the tensor product $\oid\otimes\cid$ a \emph{locally closed idempotent}.
We say that $\cT$ \emph{has enough idempotents} if every tt-ideal is generated by locally closed idempotents.
\end{defn}

\begin{lem}
\label{sta:ID-basic}%
Assume $\cT$ has enough idempotents.
Then $\Spc(\cT)$ is equal to its dual $\Spc(\cT)^*=\Spec(\RIc(\cT))$.
In particular, it is a Stone space.
\end{lem}
\begin{proof}
Let $t\in\cT$ and consider $\supp(t)\subseteq\Spc(\cT)=:X$.
By our assumption, we have $\ttid{t}=\ttid{\oid_1\otimes\cid_1,\ldots,\oid_n\otimes\cid_n}$ for some locally closed idempotents $\oid_i\otimes\cid_i$, hence
\[
\supp(t)=\bigcup_i\supp(\oid_i)\cap\supp(\cid_i)=\bigcup_iX\backslash\supp(\oid_i^\bot)\cap X\backslash\supp({}^\bot\cid_i).
\]
This shows that every quasi-compact open in $\Spc(\cT)^*$ is also quasi-compact open in $\Spc(\cT)$.
By duality, we have $\Spc(\cT)=\Spc(\cT)^*$.
In particular, every quasi-compact open is also closed hence the space is Stone.
\end{proof}

\begin{prop}
\label{sta:idempotents-spcs}%
Assume that $\cT$ has enough idempotents.
Then the map $\epsilon:\Spc(\cT)\to\Spcs(\cT)$ is a bijection and exhibits the former as the universal Stone space mapping to the latter.
\end{prop}
\begin{proof}
The fact that $\Spec(\RIc(\cT))$ is Stone (\Cref{sta:ID-basic}) says precisely that $\RIc(\cT)$ is a Boolean algebra~\cite[Corollary~II.4.4]{MR698074}.
Thus the homomorphism $\sqrt{-}:\SI(\cT)\to\RIc(\cT)$ factors uniquely through the free Boolean algebra $\beta\SI(\cT)$ on the distributive lattice~$\SI(\cT)$.
The conclusion of the proposition is precisely that the homomorphism of Boolean algebras
\[
\bar{f}:\beta\SI(\cT)\to\RIc(\cT)
\]
is bijective.
(Here, we use that $\Spc(\cT)=\Spec(\RIc(\cT))$, by \Cref{sta:ID-basic}).

For surjectivity of~$\bar{f}$, let $\cK\in\RIc(\cT)$ so that $\cK$ is the radical of a finitely generated tt-ideal.
By our assumption, we may choose these generators to be locally closed idempotents~$\oid_i\otimes\cid_i$.
Letting $\beta(\oid)\in\beta\SI(\cT)$ denote the image of the smashing ideal $\ttid{\oid}$ in $\beta\SI(\cT)$, we find that $\bar{f}$ maps
\[
\bigvee_{i}\beta(\oid_i)\wedge\beta({}^\bot\cid_i)^\bot
\]
to $\cK$.
Injectivity of~$\bar{f}$ follows from \Cref{sta:smashing-lattice,sta:free-boolean-kernel}.
\end{proof}

\begin{lem}
\label{sta:free-boolean-kernel}%
Let $f:A\into B$ be an injective homomorphism of distributive lattices with $B$ a Boolean algebra.
Assume that the induced homomorphism $\bar{f}:\beta A\onto B$ is surjective.
Then $\bar{f}$ is an isomorphism.
\end{lem}
\begin{proof}
Let $\bar{a}\in\ker(\bar{f})$.
Then $\bar{a}=\bigvee_i\beta(a_i)\wedge\beta(b_i)^\bot$ for some $a_i,b_i\in A$.
As $\bar{f}(\bar{a})=0$ we must have $\bar{f}(\beta(a_i))\wedge\bar{f}(\beta(b_i))^\bot=0$ for each~$i$, which is equivalent to
\[
f(a_i\wedge b_i)=f(a_i)\wedge f(b_i)=f(a_i).
\]
By injectivity of~$f$ we obtain $a_i\wedge b_i=a_i$ and hence $\beta(a_i)\wedge\beta(b_i)^\bot=0$ as required.
\end{proof}

\subsection{Classification}
\label{sec:smashing-ideals}

We apply the constructions and results of \Cref{sec:smashing-spectrum} to generically simple constructible systems~$\CS$.
For simplicity we will assume that all schemes $X\in\V$ are varieties over some field~$\kk$.
In addition, we will assume that $\CS$ satisfies the following `Lefschetz type' assumption.
Let $n\in\ZZ$ and $X\in\V$ a variety.
We denote by $\Hm^n(X)$ the group $\hom_{\CS(X)}(\one,\one[n])$.
\begin{itemize}
\item \namedlabel{ax:lef}{(0-Lef)} For all integral curves~$V\in\V$, $\Hm^{-1}(V)=0$ and if $U\subseteq V$ is a dense open subset then $\Hm^0(V)\to\Hm^0(U)$ is a monomorphism.
\end{itemize}

\begin{rmk}
All examples in \Cref{sec:examples} satisfy \bref{ax:lef}.
Indeed, when $X$ is regular the groups $\Hm^n(X)$ may be identified with, respectively, the Betti, $\ell$-adic, de\,Rham, absolute Hodge, and motivic cohomology of~$X$ in degree~$n$ (and weight~$0$), and these conditions are then easy properties of the cohomology theories.
Namely, the cohomology groups vanish in negative degrees and the $0$th cohomology group only depends on the number of connected components.

For Betti and $\ell$-adic cohomology such an identification continues to hold for general varieties.
For motivic sheaves the precise (expected) formula for these groups on general schemes is computed in~\cite[Proposition~11.1]{ayoub:etale-realization}.
For curves, all that is required is (besides localization) finite base change, more precisely that $f_*$ commutes with restriction along open and closed immersions for $f$ finite (namely, the normalization of the curve; see last paragraph on p.~98 in \loccit).
One either already believes that holonomic $\cD$-modules and mixed Hodge modules satisfy finite base change or it is a (lengthy but) straightforward exercise to deduce it from base change for finite morphisms between regular varieties.
\end{rmk}

The following two results are our principal findings in this section.
\begin{thrm}
\label{sta:zariski-reconstruction}%
Let $\CS$ be a generically simple constructible system on~$\V$ satisfying \bref{ax:lef}.
The composite functor
\[
\V\xto{\CS}\ttCat\op\xto{\Spcs}\TopSpc
\]
is naturally isomorphic to $X\mapsto\zar{X}$.
In particular, for every $X\in\V$ we have
\[
\Spcs(\CS(X))\cong\zar{X}.
\]
\end{thrm}

\begin{cor}
\label{sta:variety-reconstruction}%
Let $X\in\Var$ be a proper normal variety of dimension at least two and assume $k$ is uncountable algebraically closed of characteristic zero.
Let $\CS$ be a generically simple constructible system on~$X$ satisfying \bref{ax:lef}.
Then the scheme~$X$ is completely determined by the tt-category~$\CS(X)$.
\end{cor}
\begin{proof}
This follows from \Cref{sta:zariski-reconstruction} together with~\cite[Theorem~5.1.2]{kollar-lieblich-olsson-sawin:topological-reconstruction}.
\end{proof}

\begin{rmk}
\label{rmk:comparison-KLOS}%
Let us compare \Cref{sta:zariski-reconstruction} with~\cite[Proposition~5.4.5]{kollar-lieblich-olsson-sawin:topological-reconstruction}.
There, the authors reconstruct $\zar{X}$ from the \emph{abelian} category of constructible sheaves on~$X$, without the tensor structure.
Since this category forms the heart of the bounded derived constructible category, we may see the tensor structure in our work as playing the role of the t-structure in theirs.
Another difference between these two results is that we insist on characteristic zero fields as coefficients, while their argument also works with coefficients in a finite field.\footnote{Strictly speaking, their statement also imposes further restrictions on the field~$k$ and the variety~$X$.
However, these don't seem to enter the argument.}
On the other hand, our result applies to other theories besides constructible sheaves, and the reconstruction is arguably more systematic, obtained as a consequence of classification results for these tensor triangulated categories.
\end{rmk}

\begin{rmk}
\label{rmk:autoequivalences}%
Let $X\in\Var$ be as in \Cref{sta:variety-reconstruction}.
It is natural to ask to what extent the scheme~$X$ is determined by the triangulated category~$\CS(X)$ \emph{without the tensor structure}.
This is in the spirit of Bondal-Orlov's reconstruction theorem regarding the bounded derived category of coherent sheaves~\cite{MR1818984}.
Similarly, it would be interesting to study the group of autoequivalences of the triangulated category~$\CS(X)$.
\end{rmk}

\begin{rmk}
\label{rmk:universal-homeomorphisms}%
Let $f:X\to Y$ be a universal homeomorphism between algebraic $k$-varieties.
In the theories $\CS$ considered in \Cref{sec:examples}, the induced functor $f^*:\CS(Y)\to\CS(X)$ is an equivalence of tt-categories.
On the other hand, mere homeomorphisms can fail to induce equivalences.
For example, let $Y$ be the nodal curve $y^2=x^3+x^2$ over~$\CC$ and $Y^{\textup{nor}}=\affine^1_\CC$ its normalization.
Denote by $X$ the open subscheme of $Y^{\textup{nor}}$ with one of the preimages of the singular point in $Y$ removed.
The induced morphism $f:X\to Y$ is a Zariski homeomorphism (but not universally closed).
If the functor $f^*:\Dbc(Y;\QQ)\to\Dbc(X;\QQ)$ were an equivalence, its quasi-inverse would have to be~$f_*$.
However, one easily checks that $f_*\one\not\cong\one$.
\end{rmk}

The rest of the section is devoted to proving \Cref{sta:zariski-reconstruction}.
Fix $X\in\Var$ and $\CS$ a generically simple constructible system on~$X$ satisfying \bref{ax:lef}.

\begin{rmk}
\label{rmk:smashing-spectrum-outline}%
Let $\cK\subseteq\CS(X)$ be a tt-ideal.
By \Cref{sta:classification}, its support is a union of locally closed subsets $V_i$ of~$\zar{X}$, $i\in I$.
It follows that $\cK=\ttid{\one_{V_i}\mid i\in I}$ (recall \Cref{rmk:radical-automatic}).
In other words, the tt-category $\CS(X)$ has enough idempotents in the sense of \Cref{defn:enough-idempotents}.
\Cref{sta:idempotents-spcs,sta:main} then tell us that the underlying set of $\Spcs(\CS(X))$ is the set of points of~$X$ (and such that the Boolean algebra generated by its opens is the constructible topology).
So, our goal is to show that the only open idempotents in $\CS(X)$ are of the form $\one_U\to \one$ where $U\subset X$ is a Zariski open subset.
\end{rmk}
We start with the following simple observation.
\begin{lem}
\label{sta:idempotent-can-form}%
  Let $\cid$ be a closed idempotent in $\CS(X)$ with $\Supp(\cid)=U\subseteq \zar{X}$ an open subset.
Then there exists an isomorphism $\cid\cong \one_U$ such that the composite $\one_U\to\one\to\cid\cong\one_U$ is the identity.
\end{lem}
\begin{proof}
Consider the following commutative square
\[
\begin{tikzcd}
\one_U\otimes\cid
\ar[r]
&
\cid
\\
\one_U
\ar[u]
\ar[r]
&
\one
\ar[u]
\end{tikzcd}
\]
where all morphisms are the obvious ones.
By our assumption and \Cref{sta:main} we know $\supp(\one_U)=\supp(\cid)$ which implies that the left vertical arrow and the top horizontal arrow are both invertible.
\end{proof}

\begin{prop}
\label{sta:open-idempotent-open}%
Let $\oid$ be an open idempotent in~$\CS(X)$.
Then $\Supp(\oid)$ is an open subset of $\zar{X}$.
\end{prop}
\begin{proof}
We find it more convenient to work with the associated closed idempotent $\cid$.
Let $V=\Supp(\cid)$ and 
let $\bar{V}$ be the closure of~$V$.
Then $\cid_{\bar{V}}\cong\cid$ hence we may replace $X$ by $\bar{V}$ and assume that $V$ is dense in~$X$.
To show that $V=X$ we may assume $X$ is integral.
Indeed, if $X'\subseteq X$ is an irreducible component then $\cid|_{X'}$ is a closed idempotent in $C(X')$ with support $\Supp(\cid|_{X'})=V\cap X'$.

We proceed by induction on the dimension of $X$.
By constructibility, $V$ contains a dense open subset $U$ of $X$.
If $\dim(X)=0$ then $X$ is discrete and there is nothing to show.
If $\dim(X)=1$ then $V$ is a dense open, with closed complement~$Z$.
By \Cref{sta:idempotent-can-form}, we have a retraction $\one\to \one_V$ to the canonical morphism $\one_V\to \one$ and thus a section $\one_Z\to \one$ to the canonical morphism $\one\to \one_Z$.
By adjunction, this section corresponds to a morphism $\one\to i_*i^!\one$ where $i:Z\into X$ denotes the closed embedding.
Consider the long exact sequence induced by the second triangle in \Cref{recollement-bbd},
\[
\cdots\to\Hm^{-1}(V)\to\Hm^0(X,Z)\to\Hm^0(X)\to\Hm^0(V)\to\cdots,
\]
where $\Hm^0(X,Z):=\hom(\one,i_*i^!\one)$.
By \bref{ax:lef}, the first term vanishes and the penultimate arrow is a monomorphism.
We conclude that the section $\one_Z\to\one$ is zero and so $V=X$ as desired.

We now assume $\dim(X)>1$.
Since $V$ being closed is a local property we may assume $X=\Spec(A)$ is affine.
It follows that there are at most finitely many primes $\mfp_1,\ldots,\mfp_n$ of height~$1$ not in $V$.
Let $\mfm$ be a maximal ideal in $A$.
In particular, $\mfm$ cannot be contained in one of the $\mfp_i$. 
By the prime avoidance lemma, there exists $f\in\mfm\backslash\left(\mfp_1\cup\cdots\cup\mfp_n\right)$.
Choose a minimal prime $f\in\mfq\subseteq\mfm$. 
By Krull's Hauptidealsatz, $\mfq$ is of height~$1$.
Therefore it belongs to~$V$.
Let $X':=\cl(\mfq)$ and set $\cid':=\cid|_{X'}$.
Our induction hypothesis applied to $(X',\cid')$ implies that $V$ contains $\cl(\mfq)$, so in particular $\mfm\in V$.
This shows that the constructible subset~$V$ contains all maximal ideals of~$A$ and hence $V=\Spec(A)=X$ as required ($A$ is Jacobson).
\end{proof}

\begin{proof}[Proof of \Cref{sta:zariski-reconstruction}]
From \Cref{sta:main,sta:spcs-space} we deduce a natural transformation $\cons{(-)}\to\Spcs(\CS(-))$ of functors $\V\to\TopSpc$.
As explained in \Cref{rmk:smashing-spectrum-outline}, it is componentwise bijective and the only thing that remains to be shown is that the support of every open idempotent in $\CS(X)$ is Zariski open.
This was shown in \Cref{sta:open-idempotent-open}.
\end{proof}

\printbibliography[heading=bibintoc]
\end{document}